\documentclass[UTF8,oneside,10pt]{article}

\usepackage[a4paper,margin=1.3in]{geometry}

\usepackage{amsmath}           
\usepackage{amssymb}           
\usepackage{mathrsfs}          
\usepackage{eucal}             
\usepackage{amsthm}            
\usepackage{mathtools}         
\usepackage{stmaryrd}          
\usepackage{arcs}              
\usepackage{esint}             
\usepackage{extarrows}         
\usepackage{enumerate}         
\usepackage{xcolor}            
\usepackage{graphicx}          
\usepackage{tikz}              
\usepackage{tikz-3dplot}       
\usepackage{tikz-cd}           
\usepackage[all,cmtip]{xy}     
\usepackage{pgfplots}          
\pgfplotsset{compat=newest}    
\usepackage{subcaption}        
\usepackage[ocgcolorlinks,linkcolor=blue]{hyperref}
\usepackage[capitalise]{cleveref}          
\usepackage[hyperref=true,backend=biber,style=alphabetic,backref=true,url=false]{biblatex}

\usepackage{tcolorbox}         
\tcbuselibrary{most}           
\usepackage{bm}                
\usepackage{slashed}           

\usepackage[utf8]{inputenc}
\usepackage[T1]{fontenc}

\usepackage{svg}               

\usepackage{multicol}          

\usepackage{amsmath,amsthm,amsfonts,amssymb,bm}

\usepackage{fourier}

\usepackage{bbold}

\renewcommand{\emptyset}{\varnothing}

\usepackage{graphicx} 
\usepackage{float} 

\theoremstyle{plain}\newtheorem{thm}{Theorem}[section]
\theoremstyle{definition}\newtheorem{defi}[thm]{Definition}
\theoremstyle{plain}
\theoremstyle{plain}\newtheorem{cor}[thm]{Corollary}
\theoremstyle{plain}\newtheorem{lem}[thm]{Lemma}
\theoremstyle{plain}
\theoremstyle{plain}\newtheorem{pro}[thm]{Proposition}
\theoremstyle{plain}
\theoremstyle{plain}
\theoremstyle{plain}
\theoremstyle{remark}
\theoremstyle{remark}
\theoremstyle{remark}\newtheorem{nota}[thm]{Notation}
\theoremstyle{remark}
\theoremstyle{plain}
\theoremstyle{plain}
\theoremstyle{plain}

\theoremstyle{remark}\newtheorem{rmk}[thm]{Remark}
\theoremstyle{remark}
\theoremstyle{remark}\newtheorem{recall}[thm]{Recall}
\numberwithin{equation}{section}
\numberwithin{thm}{section}

\renewcommand{\Re}{\operatorname{Re}}
\renewcommand{\Im}{\operatorname{Im}}

\usepackage[backend=biber,style = alphabetic]{biblatex}
\usepackage{xcolor}

\definecolor{c1}{HTML}{88d5d2} 
\definecolor{c2}{HTML}{9c9d47} 
\definecolor{c3}{HTML}{fec842} 
\definecolor{c4}{HTML}{e97a2e} 
\definecolor{c5}{HTML}{834e71} 

\begin{document}

\begin{center}
{\Large \textbf{Systems of Curves on Non-Orientable Surfaces}}

{\Large \textbf{ }} 

{\large \textsc{Xiao} CHEN} 

{\small Tsinghua University}

{\footnotesize \emph{e-mail:} x-chen20@mails.tsinghua.edu.cn}

\end{center}

\begin{abstract}
    We show that the order of the cardinality of maximal complete $1$-systems of loops on non-orientable surfaces is $\sim |\chi|^{2}$. In particular, we determine the exact cardinality of maximal complete $1$-systems of loops on punctured projective planes. To prove these results, we show that the cardinality of maximal systems of arcs pairwise-intersecting at most once on a non-orientable surface is $2|\chi|(|\chi|+1)$.




\end{abstract}





\section{Introduction}\label{section1}

    In \cite{harvey1981boundary}, Harvey defined \emph{curve complexes} in analogy to Tits buildings for symmetric spaces. In keeping with this analogy, a type of rigidity akin to Mostow rigidity \cite{zbMATH01089297} holds for the curve complex of a surface: its automorphism group is equal to the mapping class group \cite{zbMATH01089297, MR1722024, MR3209702}. This rigidity asserts that the combinatorial geometry of curve complexes encodes all algebraic information for mapping class groups. For example, Masur-Minsky established the $\delta$-hyperbolicity \cite{masur1998geometry} of the curve complex and used it to determine word length bounds for conjugating elements between pseudo-Anosov mapping classes \cite[Theorem 7.2]{MR1791145}. Dahmani, Guirardel and Osin \cite[Theorem 2.31]{MR3589159} employed the characterisation of pseudo-Anosovs as loxodromic actions on the curve complex \cite[Proposition 3.6]{masur1998geometry} to prove there is a positive $n$ such that for \emph{any} pseudo-Anosov element $a$, the normal closure
    of $a^n$ is free and purely pseudo-Anosov, thereby resolving two of Ivanov's problems \cite{MR2264532}.  

    We define two generalisations of the curve complex. The first one is called the \emph{$k$-curve complex} (\cref{k complex}). It is a simplicial complex whose cells each correspond to (an isotopy class of) a \emph{$k$-system of loops} \cite{juvan1996systems}, which is defined as a collection of simple loops in distinct free isotopy classes pairwisely intersecting at most $k$ times. We say a $k$-system of loops is \emph{complete} if the intersection number between any two loops is always $k$, and the subcomplex of the curve complex comprised of cells corresponding to complete $k$-systems is called the \emph{complete $k$-curve complex} (\cref{complete k complex}). This is the second generalisation of the curve complex. 
    
    The dimension of a cell of the $k$-curve complex is one less than the number of loops in the corresponding $k$-system, and thanks to \cite{juvan1996systems}, we know that $k$-curve complexes are finite dimensional. We say a (complete) $k$-system of loops is \emph{maximal} if it realises a cell of the complex with the maximal dimension. This naturally raises the following much-studied question: 
    \begin{center}
    \emph{what is the cardinality of a maximal (complete) system of curves\footnote{The word "curves" here refers to both simple loops and simple arcs. Take care to distinguish this usage from the "curve" in the term "\emph{curve} complex", which refers exclusively to simple loops.}?} 
    \end{center}
    
    We have come a long way since Juvan--Malni--Mohar's initial super-exponential (in Euler characteristic) upper bounds in \cite{juvan1996systems}, and the current best results for orientable surfaces $S_g$ of genus $g$ are as follows:
    \begin{itemize}
        \item there are $2g+1$ loops in any maximal complete $1$-systems of loops \cite{malestein2014topological}. The same cardinality holds more generally for punctured surfaces (\cref{pro1}).
        \item the cardinality of maximal $k$-systems of loops is between $c_k g^{k+1}$ and $ C_k g^{k+1} \log g$ for positive $c_k$ and $C_k$ \cite{MR4034921}. Greene's results hold more generally for surfaces $S_{g,n}$ with $n$ punctures when $k$ is even or when $(k-2)g \geq n-2$. 

    \end{itemize}

    We obtain the first lower and upper bounds for the cardinality of maximal complete $1$-systems of loops on non-orientable surfaces, as well as the exact cardinality of such systems for $n$-punctured projective planes. In so doing, we improve on previous lower bounds by Nicholls--Scherich--Shneidman for general $1$-systems of loops \cite{nicholls2021large} and extend Przytycki's exact cardinality of maximal $1$-systems of arcs \cite{przytycki2015arcs} to also hold on non-orientable surfaces.


    \subsection{Main Results}

    We use the following notation throughout, let

    \begin{itemize}
        \item $F$ denote a finite-type surface, possibly non-orientable.
        \item $S_{g,n}$ denote the compact orientable surface of genus $g \geq 0$ and with $n$ boundary components.
        \item $N_{c,n}$ denote the compact non-orientable surface of genus $c \geq 1$ (i.e. with $c$ cross-caps) and with $n$ boundary components.
        \item $L$ denote a system of loops and $A$ denote a system of arcs on $F$.
        \item $\mathscr{L}(F,k)$ denote the collection of all $k$-systems of loops on $F$, and let $\mathscr{A}(F,k)$ denote the collection of all $k$-systems of arcs on $F$.
        \item $\widehat{\mathscr{L}}(F,k)$ denote the collection of all complete $k$-systems of loops on $F$, and let $\widehat{\mathscr{A}}(F,k)$ denote the collection of all complete $k$-systems of arcs on $F$.
        \item $\| \mathscr{X} \|_\infty := \max\{ \#X \mid X \in \mathscr{X} \}$, where $\#X$ is the cardinality of $X$.
        \item $\chi(F)$ denote the Euler characteristic of a surface $F$. In cases without ambiguity, we often abbreviate it as $\chi$.
    \end{itemize}

    \subsubsection{Counting Arcs}

    Przytycki's \cite{przytycki2015arcs} proved that the cardinality of $1$-systems of arcs on a hyperbolic orientable surface is $2|\chi|(|\chi|+1)$. We investigate the properties of lassos, thereby extending this result to non-orientable surfaces.

    Przytycki \cite{przytycki2015arcs} proved that the cardinality of $1$-systems of arcs on a hyperbolic orientable surface is $2|\chi|(|\chi|+1)$. We aim to extend this result to non-orientable surfaces. However, Przytycki's proof of \cite[Lemma 2.5]{przytycki2015arcs} is fundamentally based on orientation-preserving isometries of $\mathbb{H}^2$ and cannot be directly adapted to the non-orientable case. By analyzing the geometric properties of lassos at honda (\cref{lemma Lasso}), we establish a stronger theorem (\cref{part triagle embedding}) that applies to the non-orientable case and improves upon the conclusion of \cite[Lemma 2.5]{przytycki2015arcs}. As a corollary, we obtain the desired generalization of \cite[Lemma 2.5]{przytycki2015arcs} in \cref{whole triagle embedding}.

    \begin{thm}
    \label{thm4}
            Let $F$ be a non-orientable complete finite-area hyperbolic surface with at least one cusp. The maximal cardinality $\| \mathscr{A}(F,1) \|_\infty$ of $1$-systems of arcs on $F$ satisfies:
            \begin{align}
                \| \mathscr{A}(F,1) \|_\infty = 2|\chi(F)|(|\chi(F)|+1). \notag
            \end{align}
    \end{thm}

    Combining with Przytycki's result, one obtains the following corollary.

    \begin{cor}
    \label{thm4cor}
            Let $F$ be a complete finite-area hyperbolic surface with at least $1$ cusp ($F$ is allowed to be non-orientable). Then,
            \begin{align}
                \|\mathscr{A}(F,1) \|_\infty = 2|\chi(F)|(|\chi(F)|+1). \notag
            \end{align}
    \end{cor}

    \subsubsection{Counting Loops}
    
    Nicholls, Scherich, and Shneidman provided a lower bound for the cardinalities of maximal $1$-systems of loops on non-orientable surfaces \cite[Theorem A]{nicholls2021large}. Their bound is a degree-two polynomial in the genus $c$ and degree-one in the number $n$ of boundary components. Our results for complete $1$-systems in \cref{cor1} (which is a corollary to \cref{thm1} for the general statement) and for general systems of loops in \cref{thm5} improve the coefficient of the quadratic term in $c$ and increase the degree of $n$ to two.

    Furthermore, in \cref{thm2}, we provide a quadratic upper bound for the cardinalities of maximal complete $1$-systems of loops on non-orientable surfaces. Taken together, these results show that the cardinalities of maximal $1$-systems of loops on non-orientable surfaces are of quadratic order expressed in terms of the Euler characteristic, which is in sharp contrast to the linear order behavior observed for orientable surfaces.

    \begin{cor}
    \label{cor1}
        Let $F = N_{c,n} $. The maximal cardinality of complete $1$-systems of loops on $F$ (denoted by $\| \widehat{\mathscr{L}}(F,1) \|_\infty$) satisfies:
        \begin{align}
        \| \widehat{\mathscr{L}}(F,1) \|_\infty \geq { {c-1+n} \choose 2 } + c = \tfrac{1}{2} \left( (c+n)^2 - 3n - c + 2 \right).\notag
        \end{align}
    \end{cor}

    \begin{thm}
        \label{thm5}
        Let $F = N_{c,n} $. The maximal cardinality of $1$-systems of loops on $F$ (denoted by $\| \mathscr{L}(F,1) \|_\infty$) satisfies:
            \begin{align}
                \|\mathscr{L}(F,1)\|_\infty & \geq { {c-1+n} \choose 2 } + { {c-1} \choose 2 } + (c-1)(c+n-2) + 1  \notag \\
                & = 2c^2+\tfrac{1}{2}n^2+2cn-6c-\tfrac{5}{2}n+5.\notag
            \end{align}
        \end{thm}

    \begin{thm}
    \label{thm2}
    Let $F = N_{c,n} $. We have
        \begin{align}
        & \|\widehat{\mathscr{L}}(F,1)\|_\infty \leq
        \begin{cases}
            \tfrac{1}{2}n^2 -  \tfrac{1}{2}n +1 , & \text{\ if $c=1$\ } , \\
            2n^2 +n +2 , & \text{\ if $c=2$\ } , \\
            2|\chi|^2 + 2|\chi| + 1 , & \text{\ otherwise\ } . \\
        \end{cases}\notag
        \end{align}
    \end{thm}

    Combining \cref{thm2} with \cref{thm1}, which provides a quadratic lower bound, we obtain the exact count in the case of the punctured projective plane.

    \begin{cor}
    \label{cor2}
        Given a $n$-punctured projective plane $N_{1,n} $, we have
        \begin{align}
        \|\widehat{\mathscr{L}}(N_{1,n},1)\|_\infty = \tfrac{1}{2}n^2 -  \tfrac{1}{2}n +1.\notag
        \end{align}
    \end{cor}



\section{Background}\label{section2}

    \subsection{Surfaces and Curves}

    \begin{figure}[H]
    \centering
    \includegraphics[width=1.0\textwidth]{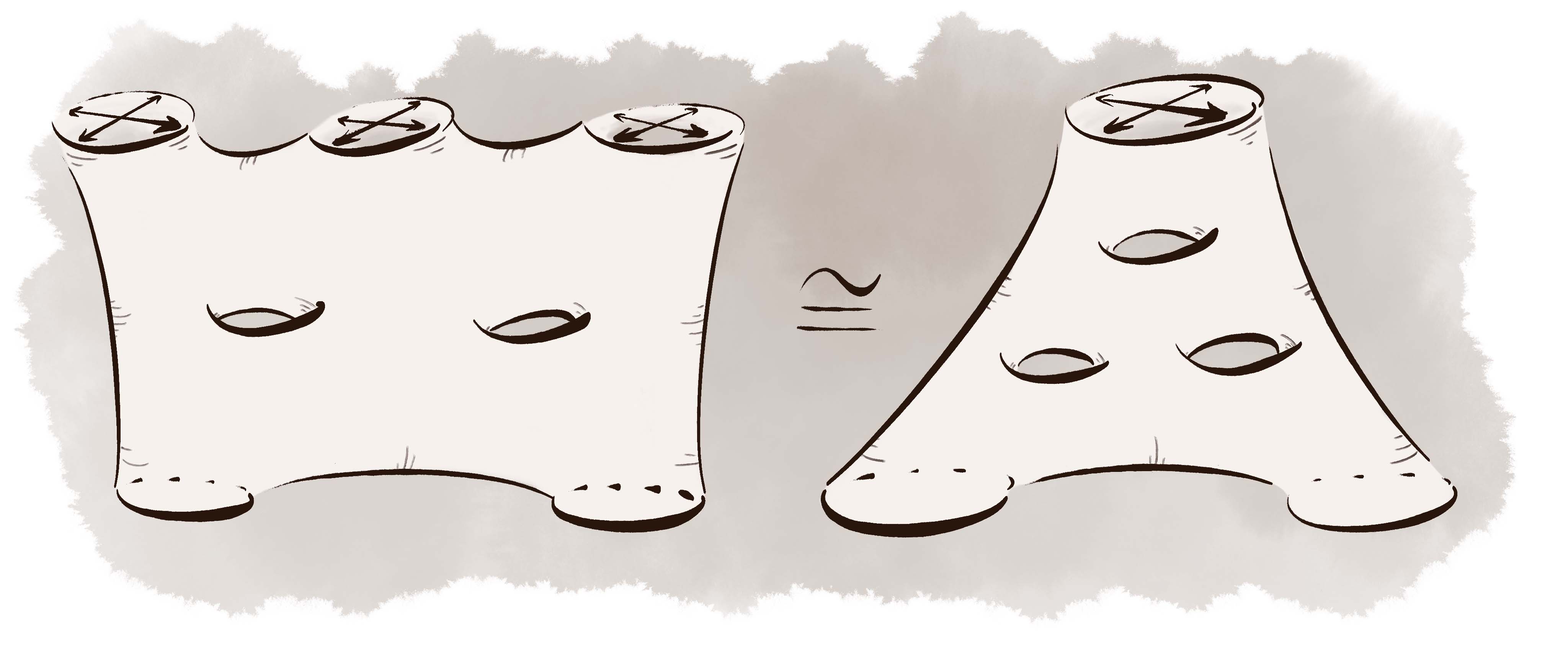}
    \caption{Two depictions of $N_{7,2}$. The $\bigotimes$ in both pictures represents a cross-cap (i.e. an $\mathbb{S}^1$ with antipodal points identified).}
    \label{TwoCrossCaps}
    \end{figure}


    \begin{defi}[Simple curves]

    We refer to both loops and arcs as curves and will often conflate a curve with its image. A \emph{simple loop} on a surface $F$ is defined by an embedding map $\gamma : \mathbb{S}^1 \hookrightarrow F$. If $F$ is a surface with non-empty boundary $\partial F$, then we define a \emph{simple arc} on $F$ as an embedding $\alpha : [0,1] \hookrightarrow F$ such that $\alpha (\partial[0,1]) \subset \partial F$.

    \end{defi}

    \begin{nota}
        We let $\gamma$ represent a loop, let $\alpha$ represent an arc and let the symbol $\beta$ represent a curve. Let $B = \mathbb{S}^1$ if $\beta$ is a simple loop and $B = [0,1]$ if $\beta$ is a simple arc.
    \end{nota}

    \begin{defi}[Regular neighbourhoods of curves]
    \label{defi RN}

    Let $\Omega$ be a finite collection of curves on $F$. A regular neighbourhood $W(\Omega)$ of $\Omega$ is a locally flat, closed subsurface of $F$ containing $\cup_{\beta \in \Omega} \beta$ such that there is a strong deformation retraction $H: W(\Omega) \times I \rightarrow W(\Omega)$ onto $\cup_{\beta \in \Omega} \beta$ where $H\vert_{(W(\Omega) \cap {\partial F}) \times I}$ is a strong deformation retraction onto $\cup_{\beta \in \Omega}   \beta \cap {\partial F}$.
    
    \end{defi}

    \begin{defi}[Free isotopies]
    
    Let $\beta_0 , \beta_1$ be two curves. We say $\beta_0$ is \emph{freely isotopic} to $\beta_1$, if there is a family of curves $\left\{ \beta(t,\cdot) \middle| t \in [0,1] \right\}$ such that $\beta(0,\cdot) = \beta_0 , \beta(1,\cdot) = \beta_1$ and $\beta(\cdot,\cdot): [0,1] \times B \rightarrow F$ is a continuous map. We denote the \emph{free isotopy class} by $H := [\beta_0] = [\beta_1]$.

    \end{defi}

    \begin{defi}[Essential curves]

    We say that a simple loop is \emph{essential} if it is not homotopically trivial, cannot be isotoped into any boundary component, and is primitive. We say that a simple arc is essential if it is not isotopic to any arc in any boundary component.

    \end{defi}

    \begin{defi}[$1$-sided vs. $2$-sided loops]
    
    We say that a simple loop $\gamma$ is \emph{$1$-sided} if the regular neighbourhood of $\gamma$ is homeomorphic to a Möbius strip. We say that $\gamma$ is \emph{$2$-sided} if the regular neighbourhood of $\gamma$ is homeomorphic to an annulus.

    Since the number of boundary components of the regular neighbourhood of a loop does not change under isotopy, $1$-sidedness and $2$-sidedness are well-defined for free isotopy classes.

    \begin{figure}[H]
    \centering
    \includegraphics[width=1.0\textwidth]{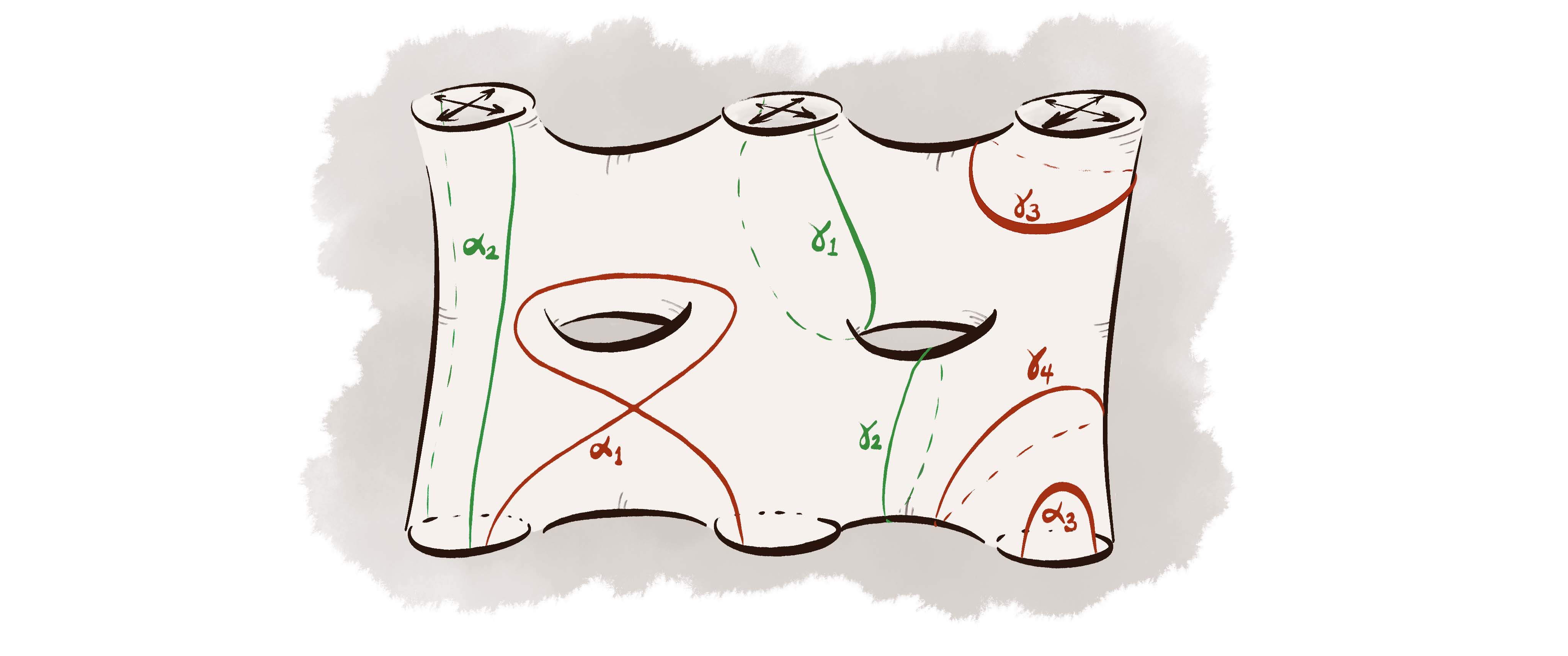}
    \caption{Here are some examples of simple essential curves (in green) and non-simple or non-essential curves (in red) on $N_{7,3}$. The curve $\alpha_1$ is a non-simple arc, $\alpha_2$ is an essential simple arc, $\gamma_1$ is an essential $1$-sided simple loop, $\gamma_2$ is a essential $2$-sided simple loop, $\alpha_3$ is a non-essential simple arc and $\gamma_3, \gamma_4$ are two non-essential simple loops, where $\gamma_3$ is non-primitive.}
    \label{CurvesExamples}
    \end{figure}
        
    \end{defi}

    \subsection{Systems of Curves}
    
    \begin{defi}[Geometric intersection numbers]

        Given two curves $\beta_1: B_1 \rightarrow F$, $\beta_2: B_2 \rightarrow F$, we define the \emph{geometric intersection number} $i (\beta_1 , \beta_2)$ as
        \begin{align}
            i (\beta_1 , \beta_2) := \# \left\{ (b_1, b_2) \in B_1 \times B_2 \middle| \beta_1(b_1) = \beta_2(b_2) \right\} . \notag
        \end{align}
    
        Consider two free isotopy classes (not necessarily distinct) $[\beta_1] , [\beta_2]$ of two simple curves $\beta_1, \beta_2$ on $F$. We define the \emph{geometric intersection number} of $[\beta_1]$ and $[\beta_2]$ as follows:
        \begin{align}
            i ([\beta_1] , [\beta_2]) := \min \left\{ i (\bar{\beta}_1 , \bar{\beta}_2) \middle|  \bar{\beta}_1 \in [\beta_1] , \bar{\beta}_2 \in [\beta_2] \right\}. \notag
        \end{align}
            
        \end{defi}
    
    \begin{rmk}
        We say two curves are \emph{transverse} if they intersect transversely at all intersection points. 
    \end{rmk}

    \begin{defi}[Systems of curves]
    
    A \emph{system} $\Omega$ of curves on $F$ is either a collection $\Omega = L$ of simple loops or a collection $\Omega = A$ of simple arcs, such that the curves in $\Omega$ are essential, any two distinct curves $\beta_1, \beta_2$ are transverse, non-isotopic, and are in minimal position (that is $i(\beta_1, \beta_2) = i([\beta_1], [\beta_2])$).
    
    \end{defi}

    \begin{defi}[$k$-systems of curves]
    
    We call a system of curves a \emph{$k$-system of curves} on $F$ if the geometric intersection number of any pair of elements in the system is at most $k$. Let $\mathscr{L}(F,k)$ be the collection of all $k$-systems of loops on $F$, and let $\mathscr{A}(F,k)$ be the collection of all $k$-systems of arcs on $F$.
    
    \end{defi}

    \begin{defi}[Complete $k$-systems of curves]
    
    We call a system of curves a \emph{complete $k$-system of curves} on $F$ if the geometric intersection number of any pair of elements in the system is exactly $k$. We let $\widehat{\mathscr{L}}(F,k)$ or $\widehat{\mathscr{A}}(F,k)$ be the collection of all complete $k$-systems of loops or of arcs on $F$.
    
    \end{defi}

    \begin{defi}[Equivalent systems]
    
    Let $\Omega = \{\beta_i\}_{i \in \mathscr{I}}, \bar{\Omega} = \{ \bar{\beta}_i \}_{i \in \bar{\mathscr{I}}}$ be respective systems of curves on $F$. We say $\Omega$ and $\bar{\Omega}$ are \emph{equivalent} if there is a automorphism $\varphi : F \rightarrow F$ and a bijection $\psi : \mathscr{I} \rightarrow \bar{\mathscr{I}}$ such that $\varphi \circ \beta_i$ is isotopic to $\bar{\beta}_{\psi(i)}$ for every $i \in \mathscr{I}$. In particular, we say $\Omega$ and $\bar{\Omega}$ are \emph{isotopic} if there is a bijection $\psi : \mathscr{I} \rightarrow \bar{\mathscr{I}}$ such that $\beta_i$ is isotopic to $\bar{\beta}_{\psi(i)}$ for every $i \in \mathscr{I}$.

    
    \end{defi}

    


    \cref{SystemEquivalence} is an example of two equivalent complete $1$-systems of loops on $S_{2,0}$.

    \begin{figure}[H]
    \centering
    \includegraphics[width=1.0\textwidth]{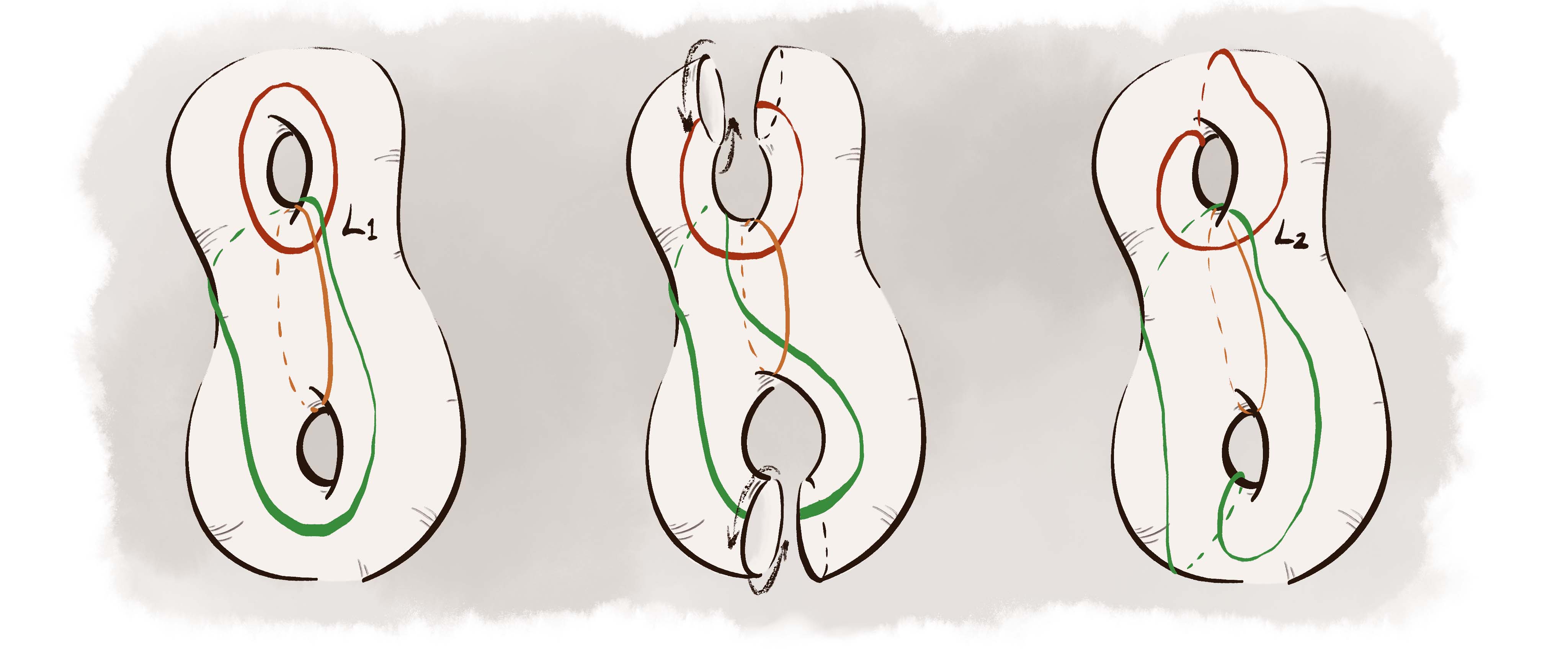}
    \caption{$L_1$ and $L_2$ are two equivalent complete $1$-systems of loops on $S_{2,0}$.}
    \label{SystemEquivalence}
    \end{figure}

    \begin{defi}[(Complete) $k$-curve complexes]
    \label{k complex}
        \ \ 
        \begin{itemize}
        \item The \emph{$k$-curve complex} $\EuScript{C}_k(F)$ is a simplicial complex defined as follows. The $0$-simplices are free isotopy class of simple loops, and $\EuScript{C}_k(F)$ has an $n$-simplex ($n \geq 1$) for every collection of $n+1$ distinct free isotopy classes of simple loops that pairwise intersect at most $k$ times.
        \item The \emph{complete $k$-curve complex} $\widehat{\EuScript{C}}_k(F)$ is a simplicial complex by defining a $0$-simplex as a free isotopy class of simple loops and an $n$-simplex ($n \geq 1$) as a collection of $n$ distinct free isotopy classes of simple loops that pairwise intersect exactly $k$ times.
        \end{itemize}
    \end{defi}

    \begin{defi}[(Complete) $k$-arc complexes]
    \label{complete k complex}
    \ \ 
    \begin{itemize}
        \item The \emph{$k$-arc complex} $\EuScript{A}_k(F)$ is a simplicial complex by defining a $0$-simplex as a free isotopy class of simple arcs and an $n$-simplex ($n \geq 1$) as a collection of $n$ distinct free isotopy classes of simple loops that pairwise intersect at most $k$ times.
        \item The \emph{complete $k$-arc complex} $\widehat{\EuScript{A}}_k(F)$ is a simplicial complex by defining a $0$-simplex as a free isotopy class of simple arcs and an $n$-simplex ($n \geq 1$) as a collection of $n$ distinct free isotopy classes of simple loops that pairwise intersect exactly $k$ times.
    \end{itemize}
\end{defi}

The following is clear:

    \begin{pro}
    $\widehat{\EuScript{C}}_k(F)$ is a subcomplex of $\EuScript{C}_k(F)$, and $\widehat{\EuScript{A}}_k(F)$ is a subcomplex of $\EuScript{A}_k(F)$.
    \end{pro}

    \begin{rmk}
        In \cite{zbMATH07704564}, the authors defined the $k$-curve graph, which is the $1$-skeleton of the $k$-curve complex from the \cref{k complex}.
    \end{rmk}

\section{Complete $1$-Systems of Loops On Orientable Surfaces}\label{section3}

    Justin Malestein, Igor Rivin, and Louis Theran \cite{malestein2014topological} showed the cardinality for maximal complete $1$-systems of loops on the orientable closed surface of genus $g \geq 1$ is $2g + 1$. To extend this result to compact orientable surfaces with non-empty boundary, we use the following lemma:

    \begin{lem}
    \label{lem0}

    If $\gamma_1, \gamma_2$ are two essential loops on a surface $F$ that transversely intersect once, then they are in minimal position. Moreover, if $\Omega$ is a collection of essential loops on an orientable surface $S$ such that every pair of loops in $\Omega$ transversely intersects once, then $\Omega$ is a complete $1$-system of loops. 
        
    \begin{proof}[Proof.]
    It suffices to prove that $i (\gamma_1,\gamma_2) = 1$ implies $i ([\gamma_1], [\gamma_2]) = 1$. Since $i (\gamma_1,\gamma_2) \geq i ([\gamma_1], [\gamma_2])$, it suffices to prove $i ([\gamma_1], [\gamma_2]) \neq 0$. If $i ([\gamma_1], [\gamma_2]) = 0$, then $\gamma_1,\gamma_2$ are not in minimal position by definition. They bound a bigon, and $\gamma_1,\gamma_2$ intersect at least two times, which is a contradicts the bigon criterion \cite{MR2850125}.

    On an orientable surface $S$, for each pair $\gamma_1, \gamma_2$ of loops which are isotopic, we have $i ([\gamma_1], [\gamma_2]) = 0$, because we can isotop these two loops to parallel positions, such that they become the two boundaries of an annulus. Since each pair $\gamma_1, \gamma_2$ in $\Omega$ has $i ([\gamma_1], [\gamma_2]) = 1$, they are not isotopic. By definition, $\Omega$ is a complete $1$-system of loops.
    \end{proof}
    \end{lem}

    \begin{pro}
    \label{pro1}
    Given $S_{g,n}$ with $g \geq 1$ or $n \geq 4$, then
    \begin{align}
    \|\widehat{\mathscr{L}}(S_{g,n},1)\|_\infty = 2g + 1 \text{.}\notag
    \end{align}
    
    \begin{proof}[Proof.]

    We first consider the case when $g=0$. If $n \leq 3$, then none of the simple loops on the surface are essential. If $n \geq 4$, then there exist essential simple loops but all of them are separating. Since any loop intersects a separating loop an even number of times, any maximal complete $1$-system of loops in $S_{0,n}$ contains only one loop. This takes care of the $S_{0,n}$ case. 

    When $g \geq 1$, we show that $\|\widehat{\mathscr{L}}(S_{g,n},1)\|_\infty \geq \|\widehat{\mathscr{L}}(S_{g,0},1)\|_\infty$. Let $L$ be a complete $1$-system of loops on $S_{g,0}$. Remove $n$ disjoint disks from $S_{g,0}$ while avoiding every $\gamma$ in $L$. The resulting loops $\bar{L}$ on $S_{g,n}$ remain transverse and pairwisely intersect once. Moreover, every loop $\bar{\gamma}$ in $\bar{L}$ is essential. If not, then $\bar{\gamma}$ will be isotopic to the boundary of a disk we have removed, and the original loop $\gamma$ then bounds a disk, so $\gamma$ is nullhomotopic, which contradicts $\gamma$ being in a $1$-system of loops. By \cref{lem0}, $\bar{L}$ is a complete $1$-system of loops on $S_{g,n}$.
    
    It remains to show that, when $g \geq 1$, $\|\widehat{\mathscr{L}}(S_{g,n},1)\|_\infty \leq \|\widehat{\mathscr{L}}(S_{g,0},1)\|_\infty$. Let $\bar{L}$ be a complete $1$-system of loops on $S_{g,n}$, whose cardinality is bigger than $1$. Fill all boundaries of $S_{g,n}$ with disks. The resulting loops $L$ on $S_{g,0}$ are still transverse and pairwisely intersect once. Moreover, every loop $\gamma$ in $L$ is essential. If not, then $\gamma$ is trivial, which means the original loop $\bar{\gamma}$ either is isotopic to a boundary of the surface $S_{g,n}$ or bounds multiple boundaries of $S_{g,n}$. Thus, $\bar{\gamma}$ is either essential or separating, which is incompatible with $\bar{\gamma}$ being in a $1$-system of loops whose cardinality is bigger than $1$. By \cref{lem0}, $L$ is a complete $1$-system of loops on $S_{g,0}$.
    
    \end{proof}
    \end{pro}

    \begin{rmk}
        \label{rmk1}
        As a intuitive explanation, we can consider the regular neighbourhood of a pair of transverse loops intersecting once, which is always a punctured torus (\cref{TwoTwoSided}). Hence the surface $S_{g,n}$ is homeomorphic to connected sum of a torus containing those two loops and a $S_{g-1,n}$ (\cref{TwoTwoSided}). One sees that the regular neighbourhood of those two loops remain unchanged even when we fill all boundaries of $S_{g,n}$.

        \begin{figure}[H]
            \centering
            \includegraphics[width=1.0\textwidth]{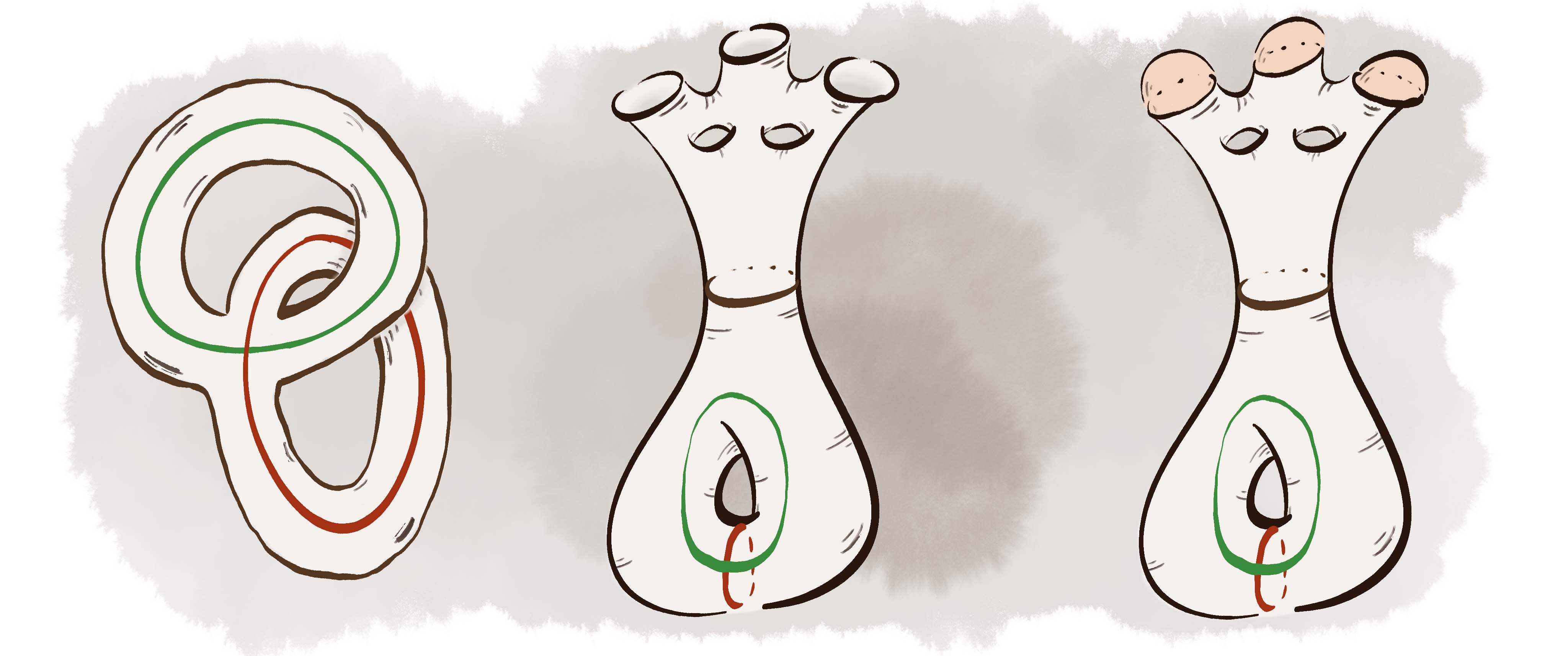}
            \caption{The regular neighbourhood of two $2$-sided loops intersecting once.}
            \label{TwoTwoSided}
            \end{figure}

        \end{rmk}

    \begin{figure}[H]
        \centering
        \includegraphics[width=1.0\textwidth]{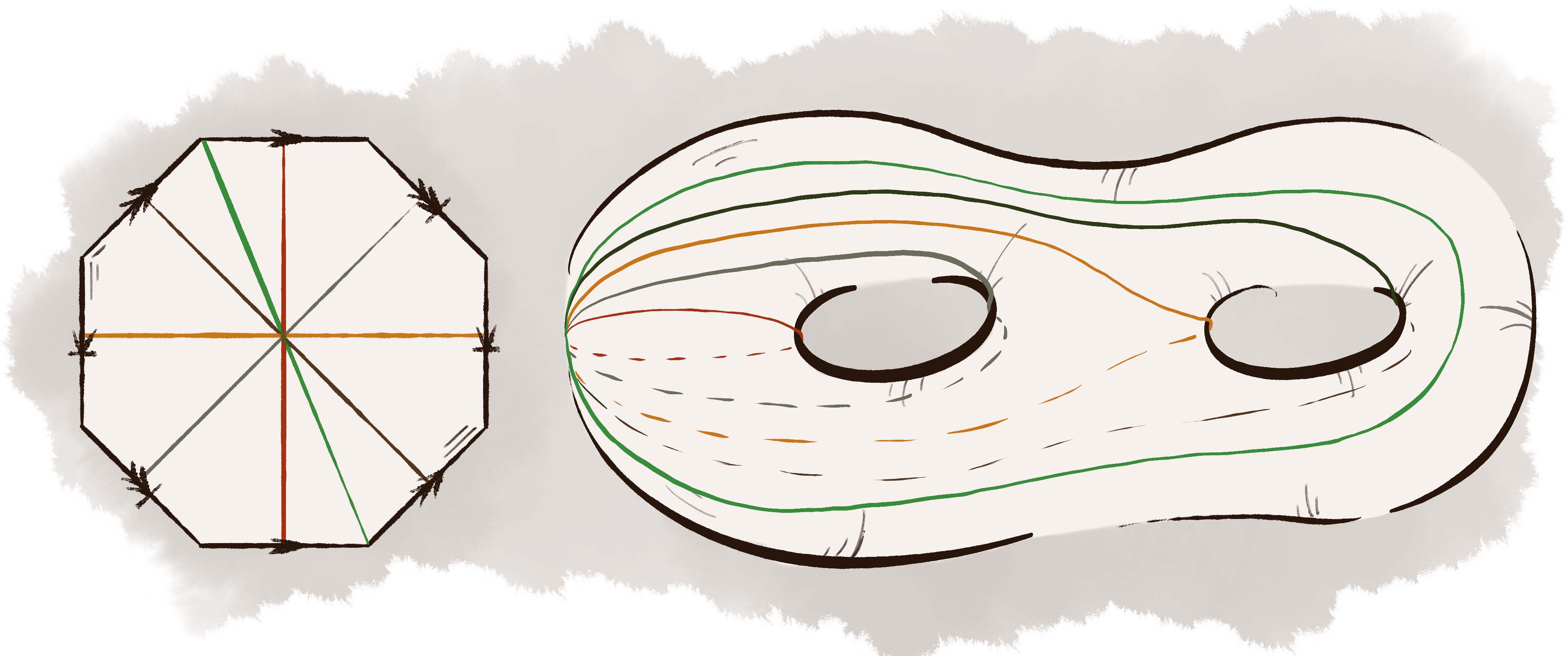}
        \caption{A maximal complete $1$-system of loops on $S_{2,0}$.}
        \label{MRT Construction}
        \end{figure}

    \begin{rmk}
        In \cite{malestein2014topological}, the authors construct one possible maximal complete $1$-system of loops for $S_{g,0}$. The authors regard $S_{g,0}$ as a $4g$-polygon with opposite sides identified. There are $2g$ simple loops connecting opposite sides of the polygon, and one simple loop connecting the diagonals of the polygon. For example, \cref{MRT Construction} shows their construction on $S_{2,0}$. We connect sum a $n$-punctured sphere with the closed surface, on which there is a maximal complete $1$-system of loops based on Malestein--Rivin--Theran's construction \cite{malestein2014topological}, to obtain a construction of maximal complete $1$-systems of loops on $S_{g,n}$ (see \cref{MaxCompleteOneSystemOrientable}). In fact, The proof of \cref{pro1} suggests more constructions of maximal complete $1$-systems of loops on $S_{g,n}$. \cref{MaxCompleteOneSystemOrientable} shows one possible maximal complete $1$-system of loops for $S_{2,4}$.
    \end{rmk}

    \begin{figure}[H]
        \centering
        \includegraphics[width=1.0\textwidth]{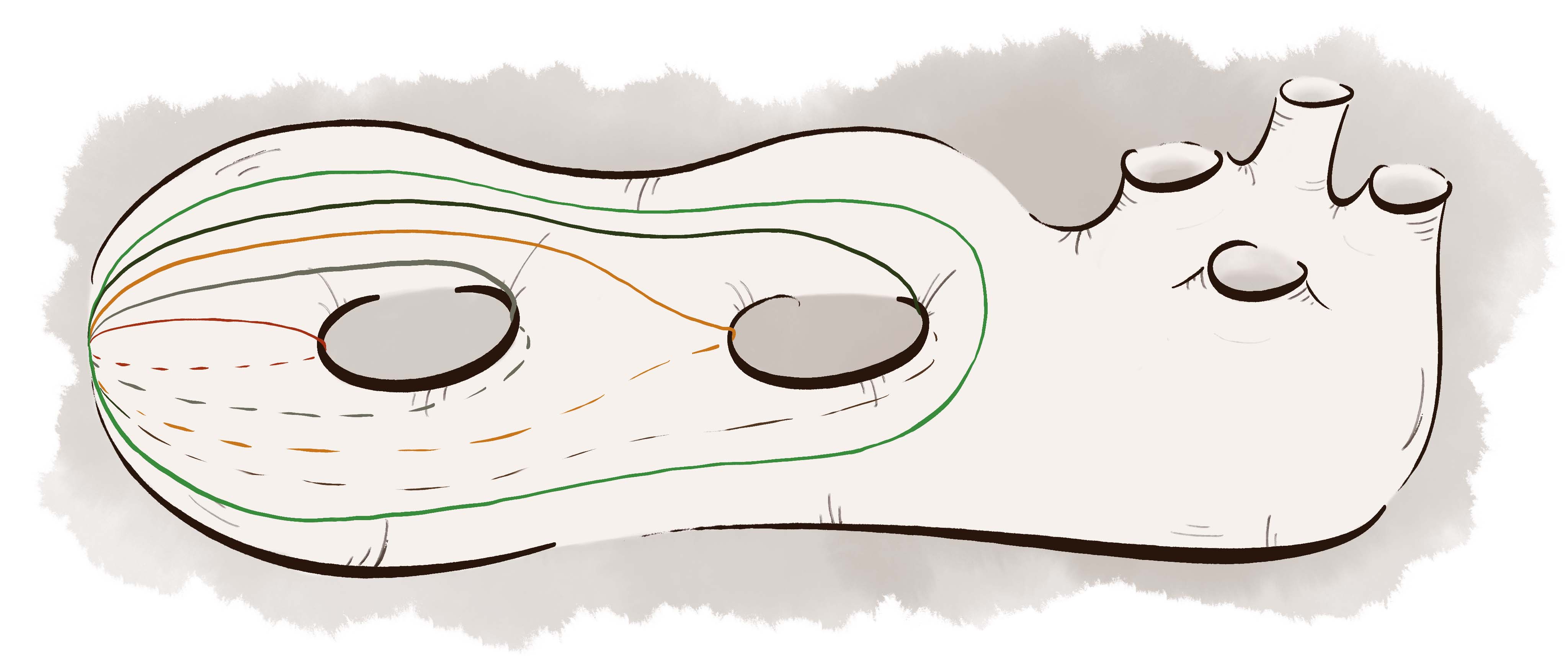}
        \caption{A maximal complete $1$-system of loops on $S_{2,4}$.}
        \label{MaxCompleteOneSystemOrientable}
        \end{figure}

\section{$1$-Systems of Arcs}\label{section4}

    Harer defined arc complexes \cite{zbMATH03951591}, which Masur-Schleimer showed are also $\delta$-hyperbolic \cite{zbMATH06168116}.
    
    In \cite{przytycki2015arcs}, Przytycki showed that the exact cardinality of maximal $1$-systems of arcs on an orientable hyperbolic surface $F = S_{g,n}$, is $2|\chi(F)|(|\chi(F)|+1)$. The main goal of this section is to extend Przytycki's theorem to non-orientable surfaces (\cref{thm4}).
    

    We replace the boundaries of the surfaces with punctures in order to utilise complete hyperbolic structures and replace arcs with geodesic arcs joining two cusps (not necessarily distinct). This change does not affect the intersection numbers between pairs of isotopy classes of arcs, because there is only one geodesic in an isotopy class \cite[Proposition 1.3]{MR2850125}, and two intersecting geodesics are always in minimal position \cite[Corollary 1.9]{MR2850125}. Therefore, the maximum number of geodesics is equal to the cardinality of the maximal system.

    \begin{defi}
        Let $F$ be a complete finite-area hyperbolic surface with at least $1$ cusp. We define a \emph{simple (geodesic) arc} on $F$ as an embedding map $\alpha : (0,1) \hookrightarrow F$ such that the image of $\alpha$ is a simple geodesic leading into cusps at both ends. Unless otherwise specified, an arc always refers to a geodesic arc in this section.
    \end{defi}

    \begin{nota}
        Let $a,b$ be two points on a hyperbolic surface (including the ideal virtices). We let $l_{ab}$ be a oriented (from $a$ to $b$) geodesic in $F$ joining $a$ and $b$.
    \end{nota}

    \begin{defi}[Tips (see \cref{TipNibSlit})]
        Let $F$ be a complete finite-area hyperbolic surface with at least $1$ cusp, and let $h$ be an oriented embedded horocycle at a cusp $p$. Let $\alpha_1 := l_{pq}, \alpha_2 := l_{pr}$ (in this order) be two \emph{oriented} simple arcs on $F$ with the same starting cusp $p$. We denote the first intersection point of $\alpha_1$ and $h$ is $a$ and the first intersection point of $\alpha_2$ and $h$ is $b$ (Typically, each $\alpha_i$ has only one intersection point with $h$, but if $\alpha_i$ emanates from $p$ and returns to $p$, there will be a first and a second intersection point along the direction of $\alpha_i$). We denote the segment on $h$ from $a$ to $b$ along the direction of $h$ as $\tau$ (If $a$ and $b$ are the same point, $\tau$ could either be $h$ itself or a point). If $\tau$ has no other intersection points with $\alpha_1$ and $\alpha_2$ except at the endpoints, then we say that $\tau$ is the {tip} of $(\alpha_1, \alpha_2)$.

    \end{defi}

    \begin{rmk}
        Tips, as currently defined, depend on the choice of orientation of $h$. However, we will later introduce "the tips of a system of arcs", which are independent of the orientation of $h$.
    \end{rmk}

    \begin{defi}[Nibs (see \cref{TipNibSlit})]
        Given a tip $\tau$ obtained from a pair $\alpha_1, \alpha_2$ of arcs on $F$. Let $\Delta_{tab}$ be an ideal triangle with three ideal vertices $t$, $a$ and $b$. A \emph{nib} of the tip $\tau$ is the unique local isometry $\nu_{\tau} : \Delta_{tab} \rightarrow F$ sending $l_{ta}$ and $l_{tb}$ to $\alpha_1$ and $\alpha_2$ (ensure that $t$ is sent to their common starting cusp) such that $\tau \subset \nu_{\tau}(\Delta_{tab})$.
    \end{defi}

    \begin{nota}
        When considering multiple tips, we will use $\Delta_\tau$ to denote the ideal triangle $\Delta_{tab}$ corresponding to $\tau$.
    \end{nota}

    \begin{defi}[Slits (see \cref{TipNibSlit})]
        \label{defi: slits}
        Given a local isometry $\nu : \Delta_{tab} \rightarrow F$. Let $n \in \Delta_{tab}$ be a point. We refer to a map $\nu|_{l_{nt}} : l_{nt} \rightarrow F$ as a \emph{slit}.
    \end{defi}
    
    \begin{nota}
        We often conflate a slit with its image, much like how we might conflate a curve with its image.
    \end{nota}

    \begin{figure}[H]
        \centering
        \includegraphics[width=1.0\textwidth]{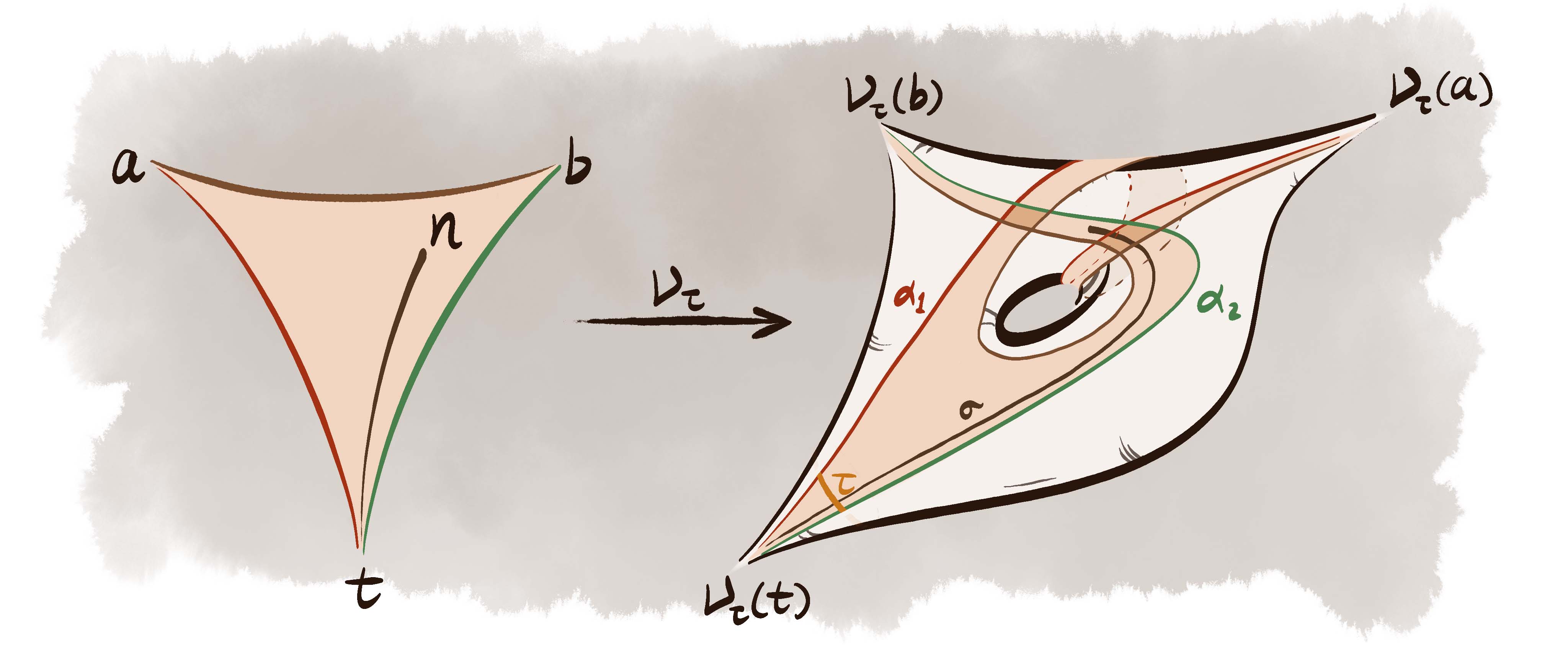}
        \caption{Depictions of a tip $\tau$, a nib $\nu_\tau$ and a slit $\nu_{\tau}|_{l_{nt}}$ on a thriced-punctured torus.}
        \label{TipNibSlit}
        \end{figure}

    \begin{defi}[Lassos (see \cref{Lassos})]
        Let $F$ be a hyperbolic surface of finite-type with at least one cusp. Given a self-intersecting geodesic arc emanating from a cusp. We define the segment of this geodesic ray from the starting cusp, up to its first self-intersection point as a \emph{lasso}. We call that first self-intersection point the \emph{honda} of the lasso. We say a lasso is \emph{$2$-sided} if its regular neighbourhood is an orientable surface (see $\sigma_1$ in \cref{Lassos}), and we say a lasso is \emph{$1$-sided} if its regular neighbourhood is a non-orientable surface (see $\sigma_2$ in \cref{Lassos}).
    \end{defi}

    \begin{defi}[Honda paths (see \cref{honda path})]
        Let $F$ be a complete hyperbolic surface of finite-type with at least one cusp, and let $\sigma$ be a lasso on $F$ emanating at the cusp $p$. A \emph{honda path} is the trajectory of the honda when we perturb the initial direction from which the lasso emanates out from the cusp.
        
        A honda path can be explicitly determined as follows: let $\widetilde{\sigma}_0$ be a lift of $\sigma$ in the hyperbolic plane. Since the group of Möbius transformations is thrice transitive, one may assume that $\widetilde{\sigma}_0$ is a vertical ray emanating from $\infty$ (which is a lift of $p$), and equation for $\widetilde{\sigma}_0$ is $\Re(z) = 1$. We denote the lift of the end point of $\sigma_0$ on $\widetilde{\sigma}_0$ by $z_0$ (see the left picture in \cref{honda path}).

        Let $\epsilon$ be a sufficiently small positive number. Then, $\widetilde{\sigma}_d, d \in (-\epsilon, +\epsilon)$ forms a collection of lifts of geodesic rays emanating from $\infty$, which terminate upon meeting $\widetilde{\sigma}_0$ and satisfy the image of $\widetilde{\sigma}_d$ lies on the vertical line $\{ z \mid \Re(z) = 1 - d \}$. Then we project $\widetilde{\sigma}_d$ to the surfaces, denoted by $\sigma_d$. Similarly, we denote the lift of the end point of $\sigma_d$ on $\widetilde{\sigma}_d$ by $z_d$.
        The honda path $L$ of $\sigma$ is the curve $\{\pi (z_d) \mid d \in (-\epsilon , +\epsilon)\}$, where $\pi : \mathbb{H}^2 \rightarrow F$ is the universal map (see the right picture in \cref{honda path}).
    \end{defi}

    \begin{figure}[H]
        \centering
        \includegraphics[width=1.0\textwidth]{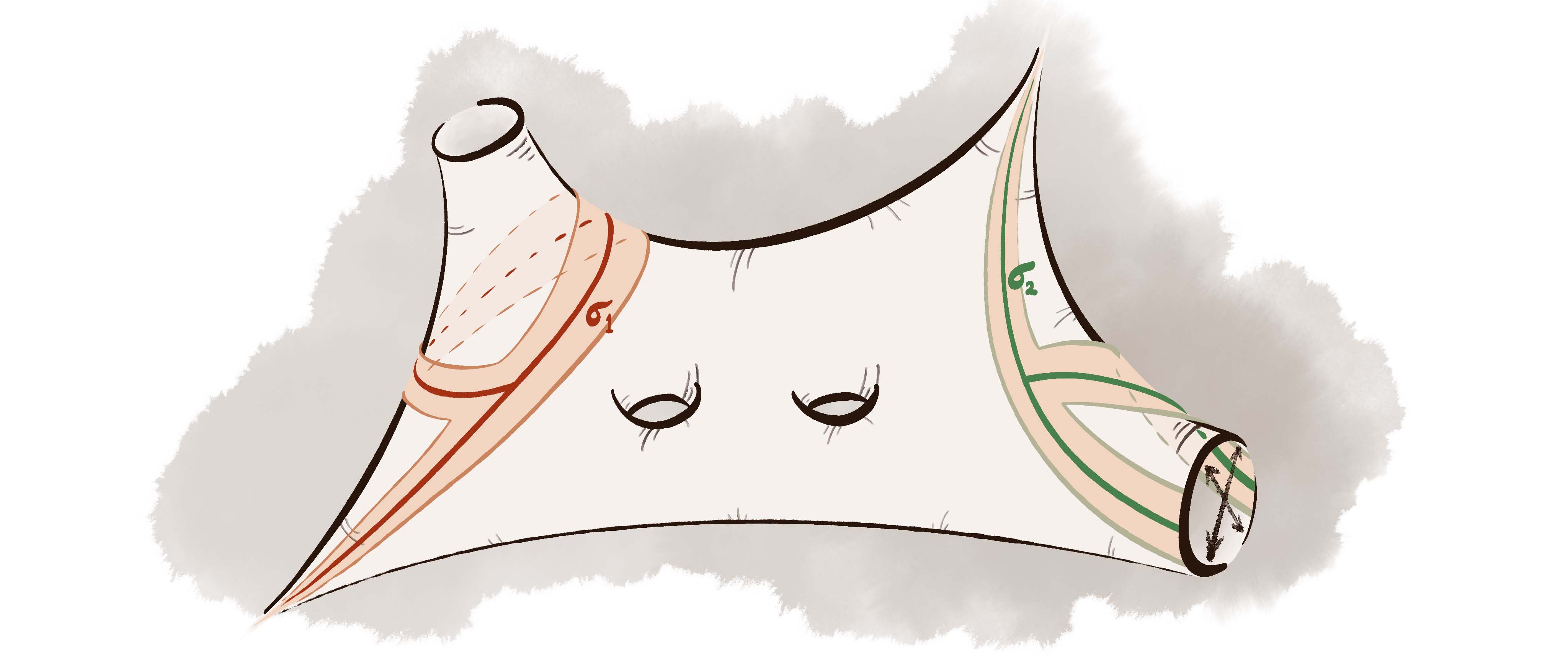}
        \caption{The the $2$-sided lasso $\sigma_1$ (left) and the $1$-sided lasso $\sigma_2$ (right).}
        \label{Lassos}
        \end{figure}

    \begin{figure}[H]
        \centering
        \includegraphics[width=1.0\textwidth]{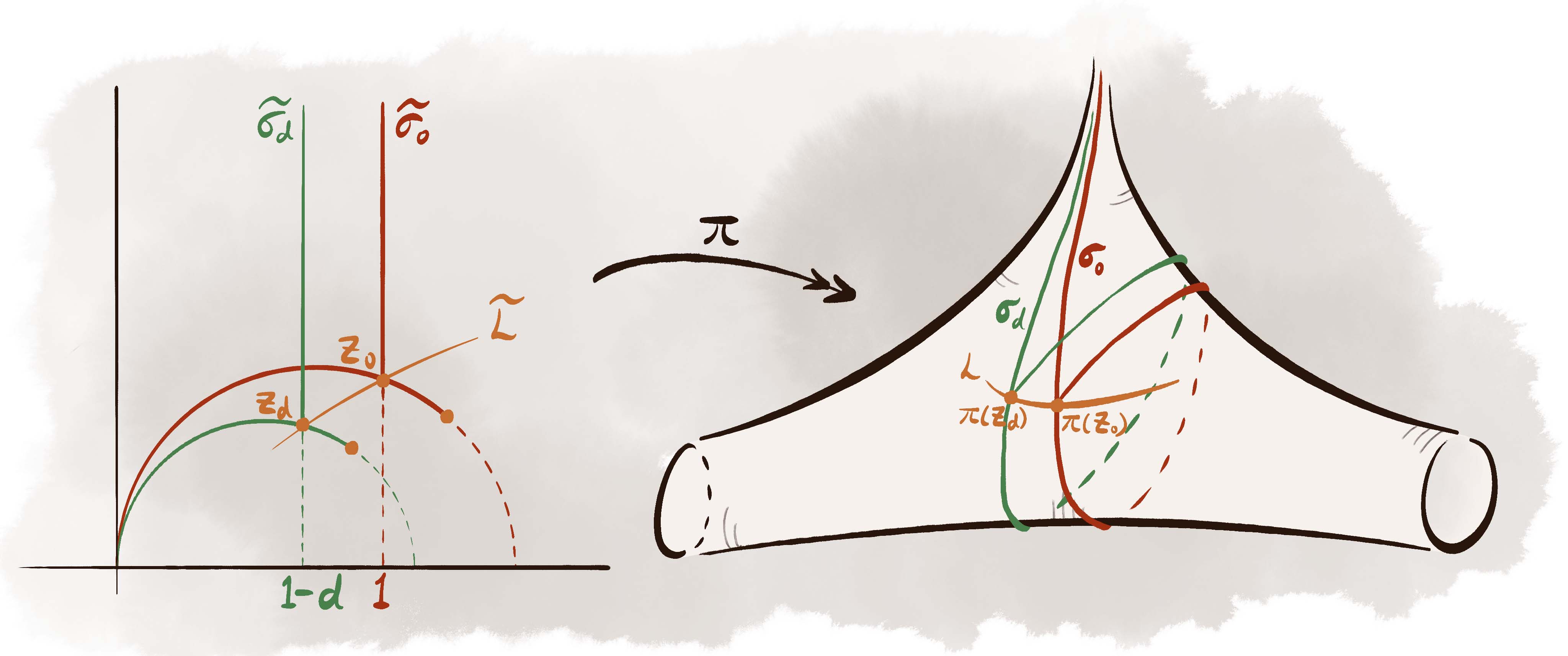}
        \caption{The honda path $L$ of $\sigma$.}
        \label{honda path}
        \end{figure}

\begin{lem}
\label{lemma Lasso}
Let $F$ be a hyperbolic surface of finite-type with at least one cusp, $\sigma$ be a lasso emanating from the cusp $p$, $L$ be the honda path of $\sigma$ and $T$ be the tangent geodesic of $L$ at the self-intersection $q$ point of $\sigma$. Denote a sufficiently small neighbourhood of $q$ by $U(q)$, then $T \cap U(q)$ separates $U(q)$ into two components $U_1(q)$ and $U_2(q)$ so that $l_{pq} \cap U(q)$ and $L \cap U(q)$ are in the same component.
\end{lem}
\begin{proof}
We first consider the case (i) that $\sigma$ is $2$-sided. Let $\varphi(\widetilde{\sigma}_0)$ be another lift of $\sigma$ which intersects $\widetilde{\sigma}_0$ at $z_0$. Here $\varphi : \mathbb{H}^2 \rightarrow \mathbb{H}^2$ is a Möbius transformation induced by an element of $\pi_1 (F)$. Since the group of Möbius transformations is thrice transitive, one may assume that $\widetilde{\sigma}_0$ emanates from $\infty$ (which is a lift of $p$) and ends at $1$, and $\varphi(\widetilde{\sigma}_0)$ emanates from $0$. We denote the end point of $\varphi(\widetilde{\sigma}_0)$ by $2c$. Since $\varphi(\widetilde{\sigma}_0)$ intersects $\widetilde{\sigma}_0$, we have $c > \tfrac{1}{2}$. The equation for $\widetilde{\sigma}_0$ is $\Re(z) = 1$, and the equation for $\varphi(\widetilde{\sigma}_0)$ is $(\Re(z)-c)^2+\Im(z)^2 = c^2$. Let $z_0 = 1 + iy_0$, then we have $y_0 = \sqrt{2c-1}$.

Since $\sigma$ is $2$-sided, $\varphi$ is an orientation-preserving non-elliptic Möbius transformation and has either one or two fixed points on the real axis.

Since $\varphi(\infty) = 0$ and $\varphi(1) = 2c$, the map $\varphi$ is of the form
\begin{align}
\varphi(z) = \tfrac{2c(1+t)}{z+t}, \quad \text{for some } t \in \mathbb{R} \setminus \{-1\}.
\end{align}
The inverse of $\varphi$ is $\varphi^{-1}(z) = -t + 2c(1+t) z^{-1}$.

Let $\widetilde{\sigma}_d$ be another geodesic ray that lies on the vertical line ${ z \mid \Re(z) = 1 - d }$ (see the left picture in \cref{Curvature}). Set $z_d$ as the intersection point of $\widetilde{\sigma}_d$ and $\varphi(\widetilde{\sigma}_d)$, then $\widetilde{L} := \left\{z_d \mid d \in (-\epsilon , +\epsilon) \right\}$ is a lift of $L$. The equations for $\widetilde{L}$ are:
\begin{align}
\label{linear eq}
\begin{cases}
z_d \in \widetilde{\sigma}_d\\
z_d \in \varphi(\widetilde{\sigma}_d)\\
\end{cases}
\Rightarrow 
\begin{cases}
z_d \in \widetilde{\sigma}_d\\
\varphi^{-1}(z_d) \in \widetilde{\sigma}_d\\
\end{cases}
\Rightarrow
\begin{cases}
\Re(z_d) = 1-d\\
\Re(-t + 2c(1+t)z^{-1}) = 1-d.\\
\end{cases} 
\end{align}
We let $z_d = x_d + iy_d$, \cref{linear eq} becomes
\begin{align}
\begin{cases}
x_d = 1-d\\
|z_d|^2 (d-(1+t))+2c(1+t) x_d =0\\
\end{cases}
\Rightarrow 
\begin{cases}
\label{eq x}
x_d = 1-d\\
|z_d|^2 = \tfrac{2c(1+t)(d-1)}{d-(1+t)}.\\
\end{cases}
\end{align}
Hence, we have
\begin{align}
\label{eq y}
y_d = \sqrt{ |z_d|^2 - x_d^2} = \sqrt{ \tfrac{2c(1+t)(d-1)}{d-(1+t)} -(1-d)^2}.
\end{align}

We compare the Euclidean curvature of $\widetilde{L}$ at $z_0$ with the Euclidean curvature of the tangent geodesic $\widetilde{T}$ of $\widetilde{L}$ at $z_0$ for determining the curved direction of $L$ under the hyperbolic metric. The formula of the Euclidean curvature of $\widetilde{L}$ is
\begin{align}
\kappa_{\widetilde{L}} := \left. \tfrac{x_d'y_d''-y_d'x_d''}{((x_d')^2+(y_d')^2)^{3/2}} \right\vert_{d=0}. \notag
\end{align}
And the Euclidean curvature of $\widetilde{T}$ is
\begin{align}
\kappa_{\widetilde{T}} := \left. \tfrac{-x_d'}{y_0 ((x_d')^2+(y_d')^2)^{1/2}} \right\vert_{d=0}. \notag
\end{align}
We explicitly calculate the ratio $R(t,c)$ of the two curvatures:
\begin{align}
R(t,c) := \tfrac{\kappa_{\widetilde{L}}}{\kappa_{\widetilde{T}}} = \tfrac{2+ ct^2 + 4ct}{2+ct^2+2t}.
\end{align}

\begin{figure}[H]
\centering
\includegraphics[width=1.0\textwidth]{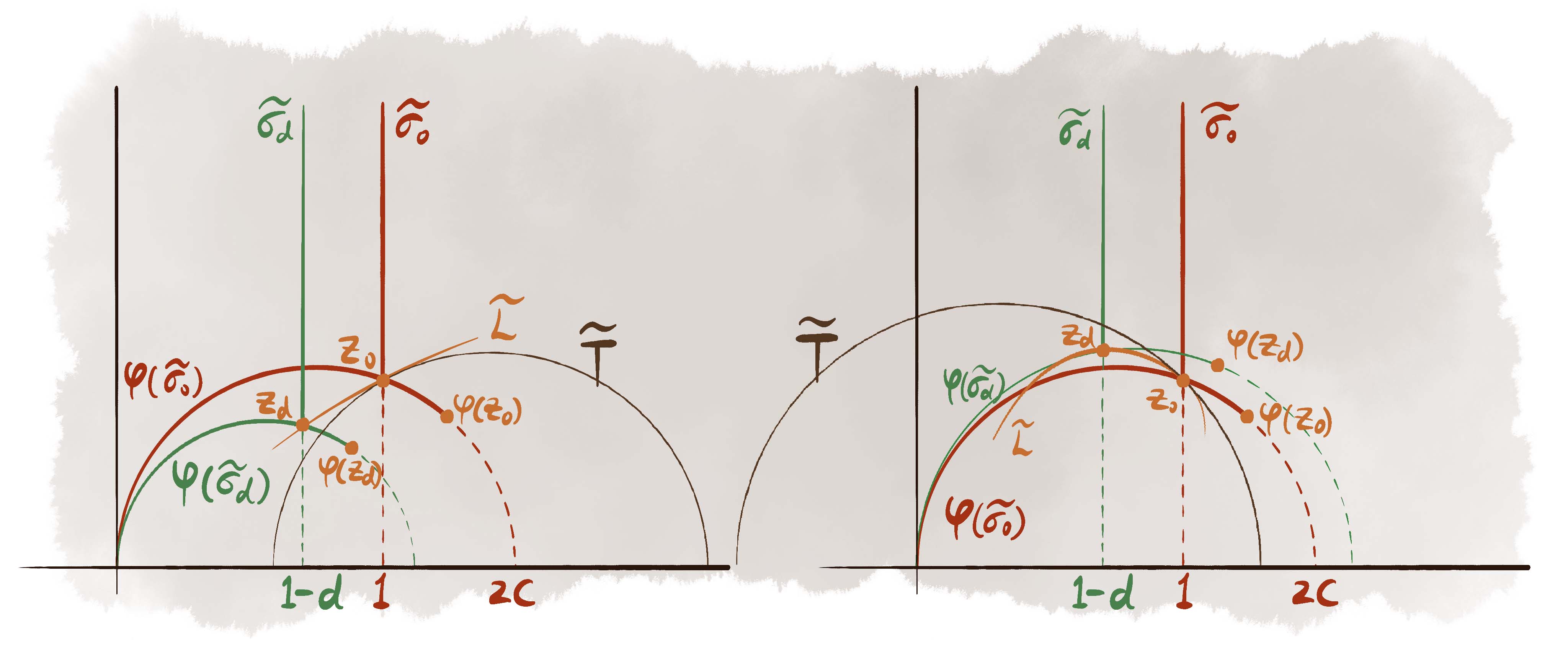}
\caption{Comparing the Euclidean curvatures of the honda path and the tangent geodesic.}
\label{Curvature}
\end{figure}

Our immediate goal is to prove $R(t,c)$ is smaller than $1$. To do so, we invoke constraints arising from the geometric properties of the lasso:

\begin{itemize}
\item The geodesic $\varphi(\widetilde{\sigma}_0)$ intersects $\widetilde{\sigma}_0$. Hence, $c>\tfrac{1}{2}$;

\item The endpoint $z_0$ of the geodesic ray $\widetilde{\sigma}_0$ (i.e., the lift of the honda) lies on the geodesic ray $\varphi(\widetilde{\sigma}_0)$, rather than the other way around (i.e., the endpoint of $\varphi(\widetilde{\sigma}_0)$ lying on $\widetilde{\sigma}_0$). Thus, $\Re(\varphi(z_0)) > \Re(z_0) \Rightarrow \Re(\tfrac{2c(1+t)}{1+y_0 i+t}) > 1 \Rightarrow \tfrac{2c(1+t)^2}{(1+t)^2+y_0^2}>1 \Rightarrow (2c-1)(1+t)^2 > y_0^2 = 2c-1 \Rightarrow t^2+2t>0 \Rightarrow t< -2 $ or $t>0$;

\item The map $\varphi$ preserves the upper half-plane, hence, $\Im(\varphi(z_0))>0 \Rightarrow \Im(\tfrac{2c(1+t)}{1+y_0 i+t}) > 0 \Rightarrow \tfrac{2c(1+t)(-y_0)}{(1+t)^2+y_0^2}>1 \Rightarrow 1+t<0  \Rightarrow t < -1$;

\item The map $\varphi$ has at least one fixed point on the real axis. Therefore, $\varphi(x) = x$ has a real solution, i.e. $x^2 +tx -2c(1+t)=0$ has a real solution. Thus, we have $t^2 + 8c(1+t) \geq 0$.
\end{itemize}

According to these conditions, we have $t < -2$ and $c > \tfrac{1}{2}$, so $R(t,c) = \tfrac{2+ ct^2 + 4ct}{2+ct^2+2t} = 1+\tfrac{2(2c-1)}{ct+\tfrac{2}{t}+2}$, which implies $ct+\tfrac{2}{t} \leq -2\sqrt{2c} < -2$. Hence, $ct+\tfrac{2}{t} +2 < 0 \Rightarrow R(t,c) < 1$. Consequently, the curvature of $\widetilde{T}$ is greater than that of $\widetilde{L}$, which implies that in a small neighbourhood of $z_0$, $l_{0z_0}$ and $\widetilde{L}$ lie on the same side of $\widetilde{T}$ (see the left picture of \cref{Curvature}). As a result, in a small neighbourhood of the self-intersection point $q$, $l_{pq}$ and $L$ lie on the same side of $T$. This completes the proof of case (i).

We next consider the case (ii) that $\sigma$ is $1$-sided. In this case, $\varphi$ is orientation-reversing. The right picture of \cref{Curvature} shows the universal cover and the lift of the geodesic rays when $\varphi$ is orientation-reversing.

Since $\varphi$ is orientation-reversing so that $\varphi(\infty) = 0$, $\varphi(1) = 2c$ , it is of the form
\begin{align}
\label{NonOri eq.}
\varphi(z) = \tfrac{2c(1+t)}{\bar{z}+t},
\end{align}
And the inverse of $\varphi$ is $\varphi^{-1}(z) = -t + 2c(1+t) \bar{z}^{-1}$.

We can similarly determine the equation for $\widetilde{L}$, to obtain a path formally identical to \cref{eq x} and \cref{eq y}:
\begin{align}
z_d = x_d + i y_d = (1-d) +i \sqrt{ \tfrac{2c(1+t)(d-1)}{d-(1+t)} -(1-d)^2}. \notag
\end{align}
Hence, the ratio of two curvatures is also formally the same, which is $R(t,c) = \tfrac{2+ ct^2 + 4ct}{2+ct^2+2t}$.

Our goal this time is to prove $R(t,c)$ is bigger than $1$. We again make use of constraints on $t$ and $c$ arising from the geometric properties of the lasso:

\begin{itemize}
\item The geodesic $\varphi(\widetilde{\sigma}_0)$ intersects $\widetilde{\sigma}_0$, hence, $c>\tfrac{1}{2}$;

\item The endpoint $z_0$ of the geodesic ray $\widetilde{\sigma}_0$ lies on the geodesic ray $\varphi(\widetilde{\sigma}_0)$, which implies $t< -2$ or $t>0$;

\item The map $\varphi$ preserves the upper half-plane, which also implies $t> -1$;

\item Since all orientation-reversing Möbius transformations $\varphi$ can be a representation of an element of $\pi_1(F)$, there is no fourth restriction as with case (i).

\end{itemize}

According to these conditions, we have $t > 0$ and $c > \tfrac{1}{2}$, so $R(t,c) = \tfrac{2+ ct^2 + 4ct}{2+ct^2+2t} = 1+\tfrac{2(2c-1)}{ct+\tfrac{2}{t}+2} > 1$. Consequently, the curvature of $\widetilde{L}$ is greater than that of $\widetilde{T}$, which implies that in a small neighbourhood of $z_0$, $l_{0z_0}$ and $\widetilde{L}$ lie on the same side of $\widetilde{T}$ (see the right picture of \cref{Curvature}). As a result, in a small neighbourhood of the self-intersection point $q$, $l_{pq}$ and $L$ lie on the same side of $T$. This completes the proof of case (ii).
\end{proof}

\begin{thm}
\label{part triagle embedding}
Let $F$ be an hyperbolic surface of finite-type with at least one cusp, and let $\Delta_{tab}$ be an ideal triangle. Let $\bar{a}$ and $\bar{b}$ be two points on $l_{ta}$ and $l_{tb}$, respectively (see \cref{ConeTriangle}). Let $\nu: \Delta_{tab} \rightarrow F$ be a local isometry such that $\nu|_{l_{t\bar{a}}}$ and $\nu|_{l_{t\bar{b}}}$ are embeddings and $\nu$ maps $t$ to a cusp. Let $n$ be a point in the triangle $\Delta_{t\bar{a}\bar{b}}$, then the slit $\nu|_{l_{tn}}$ is an embedding.
\end{thm}
\begin{proof}
Let $a_x$ be a point on $l_{t\bar{a}}$ so that the length of $l_{\bar{a} a_x}$ is $x$, and $b_x$ be a point on $l_{t\bar{b}}$ so that the length of $l_{\bar{b} b_x}$ is $x$. There is a collection of parallel geodesic segments $\{ l_{a_x b_x} \}_{x \in \mathbb{R}_{\geq 0}}$ which foliate the triangle $\Delta_{t\bar{a}\bar{b}}$ (see \cref{ConeTriangle}). Since the regular neighbourhood of a cusp is an annulus, there is a embedded horocycle $h$ around the cusp. We take a sufficiently big $x=x_0$ such that $\nu(\Delta_{t a_{x_0} b_{x_0}})$ is in the region enclosed by this horocycle, then $\nu|_{\Delta_{t a_{x_0} b_{x_0}}}$ is an embedding. Hence, for every $n \in \Delta_{t a_{x_0} b_{x_0}}$, the slit $\nu|_{l_{nt}}$ is an embedding.

We assume that there exists a slit of $\Delta_{t\bar{a}\bar{b}}$ which is not an embedding, there exists the minimal $x=x_1 \in (0 , x_0)$, such that for every $n \in \operatorname{Int}(\Delta_{t a_{x_1} b_{x_1}})$ (Here, $\operatorname{Int}(S)$ means the interior of a subset $S$), the slit $\nu|_{l_{nt}}$ is an embedding.

\begin{figure}[H]
\centering
\includegraphics[width=1.0\textwidth]{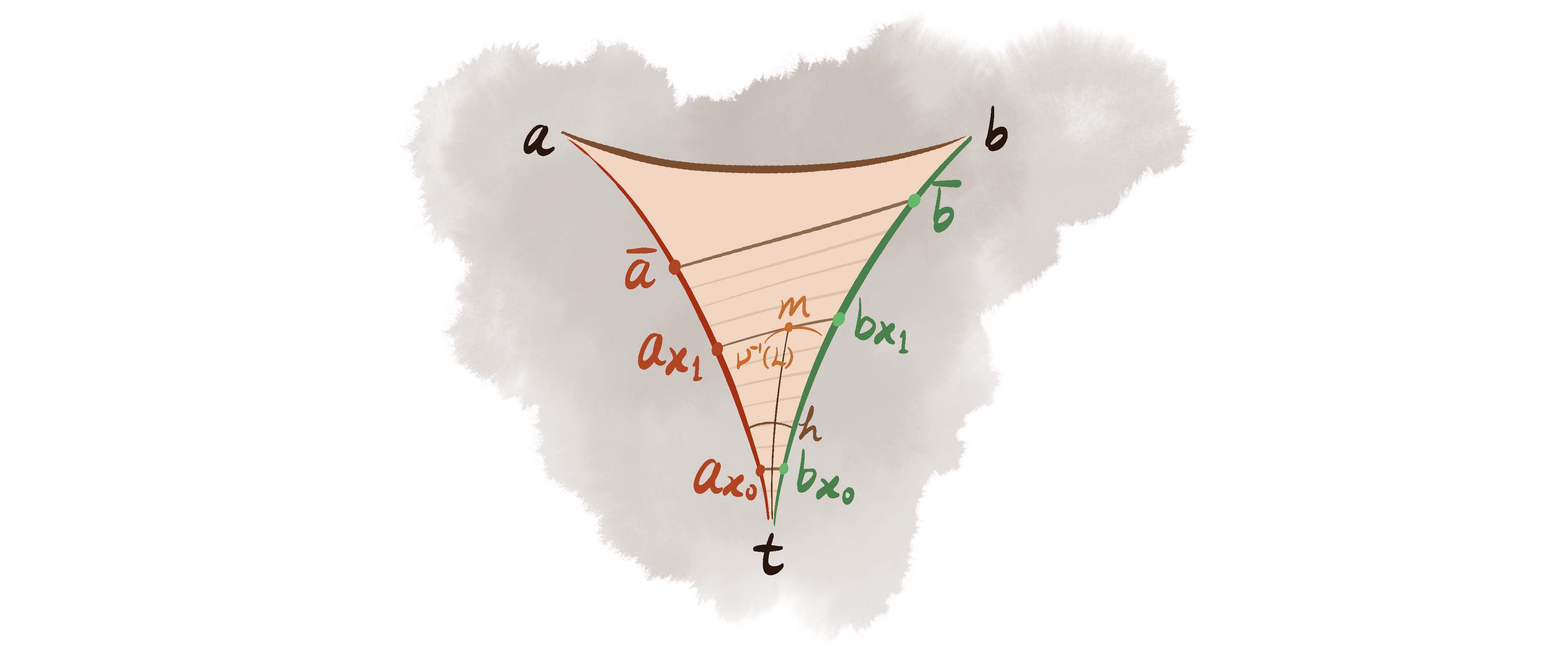}
\caption{The schematic picture of \cref{part triagle embedding}.}
\label{ConeTriangle}
\end{figure}

Hence, there is a $m \in \operatorname{Int}(l_{a_{x_1} b_{x_1}})$ such that $\sigma := \nu(l_{mt})$ is a lasso. Consider the honda path $L$ of $\sigma$. Since $x_1$ is the minimal number such that for every $n \in \operatorname{Int}(\Delta_{t a_{x_1} b_{x_1}})$, the slit $\nu|_{l_{nt}}$ is an embedding. The geodesic $\nu(l_{a_{x_1} b_{x_1}})$ on $F$ is tangent to $L$ at $\nu(m)$. By \cref{lemma Lasso}, we know $\nu^{-1}(L) \cap \operatorname{Int}(\Delta_{t a_{x_1} b_{x_1}}) \neq \emptyset$, which contradicts $\nu|_{l_{nt}}$ being embeddings for $n \in \operatorname{Int}(\Delta_{t a_{x_1} b_{x_1}})$.

\end{proof}

\begin{cor}
\label{whole triagle embedding}
Let $F$ be a hyperbolic surface of finite-type with at least one cusp, $\Delta_{tab}$ be an ideal triangle. Let $\nu: \Delta_{tab} \rightarrow F$ be a local isometry such that $\nu|_{l_{ta}}$ and $\nu|_{l_{tb}}$ are embeddings and $\nu$ maps $t,a,b$ to cusps (not necessarily distinct). Then every slit of $\nu$ is an embedding.
\end{cor}
\begin{proof}

We arbitrarily choose two respective points $\bar{a}, \bar{b}$ on $l_{ta}, l_{tb}$. There is a collection of parallel geodesic segments $\{ l_{a_x b_x} \}_{x \in \mathbb{R}}$ which foliate the triangle $\Delta_{tab}$ such that $a_0 = \bar{a}$, $b_0 = \bar{b}$ and the oriented lengths of both $l_{\bar{a}a_x}$ and $l_{\bar{b}b_x}$ are $x$.

Same as the theorem before, we may find an $x_0$ such that for every $n \in \Delta_{t a_{x_0} b_{x_0}}$, the slit $\nu|_{l_{nt}}$ is an embedding.

We assume that there is an slit which is not an embedding, then we may find the minimal $x=x_1 \in (0 , x_0)$ such that for every $n \in \operatorname{Int}(\Delta_{t a_{x_1} b_{x_1}})$, the slit $\nu|_{l_{nt}}$ is an embedding. Hence, there is a point $m$ on $\operatorname{Int}(l_{a_{x_1} b_{x_1}})$ such that $\sigma_0 := \nu(l_{mt})$ is a lasso, which contradicts to \cref{part triagle embedding}.

\end{proof}

        \begin{defi}[The tips of a system of arcs]
            Given a complete finite-area hyperbolic surface $F$ with at least one cusp, let $A$ be a system of arcs on $F$ with finite cardinality. For each cusp from which an arc in $A$ emanates, we take an oriented embedded horocycle, denoted as $h_1, \cdots, h_m$. We then consider all ordered pairs $(\alpha_1, \alpha_2)$ of \emph{oriented} arcs in $A$ such that $\alpha_1, \alpha_2$ have a common starting cusp, and the tip $\tau$ of $(\alpha_1, \alpha_2)$ does not intersect with any other arc in $A$ except at the endpoints. We refer to the set of tips $\mathcal{T}_{A}$ that satisfy this condition as the \emph{tips of $A$}.
        \end{defi}
    
        \begin{rmk}
            From this definition, we see that $\bigcup_{\tau \in \mathcal{T}_{A}} \tau = \bigcup_{i=1}^{m}h_i$. Therefore, regardless of the orientation chosen for each horocycle, all segments of these horocycles cut by the arcs in $A$ will become tips in $\mathcal{T}_{A}$. Thus, $\mathcal{T}_{A}$ is independent of the chosen orientation of the horocycles.
        \end{rmk}

    \begin{lem}
    \label{lem2}
        Let $F$ be a hyperbolic surface of finite-type with at least one cusp, $A$ be a system of arcs on $F$, and $\tau$ be a tip of $A$, then the slit of $\tau$ is an embedding.
    \end{lem}
    \begin{proof}
        We obtain this lemma directly by \cref{whole triagle embedding}.
    \end{proof}

    The following results, \cref{lem1} and \cref{lem3}, are stated greater generality than their counterparts in Przytycki \cite{przytycki2015arcs}, as the original proofs also apply for non-orientable surface.

    \begin{lem}[Przytycki {\cite[Lemma 2.3]{przytycki2015arcs}}]
    \label{lem1}
        Let $F$ be a complete finite-area hyperbolic surface with at least $2$ cusp. For each nontrivial partition $P = P_1 \sqcup P_2$ of cusps of $F$, we define $\mathscr{A}(F,0,P_1, P_2)$ as the set of $0$-systems of arcs in $F$ which have one end point in $P_1$, and the other one in $P_2$. Then we have
        \begin{align}
            \| \mathscr{A}(F,0,P_1, P_2)\|_\infty = 2|\chi (F)|.\notag
        \end{align}
    \end{lem}


    \begin{lem}[Przytycki {\cite[Lemma 2.6]{przytycki2015arcs}}]
    \label{lem3}
     Let $F$ be a complete finite-area hyperbolic surface with at least one cusp, and let $A \in \mathscr{A}(F,1)$ be a $1$-system of arcs on $F$. We Let $\Delta$ be the disjoint union of $\Delta_{\tau}$ for all tips $\tau \in \mathcal{T}_{A}$. Let $\nu : \Delta \rightarrow F$ be a map so that $\nu|_{\Delta_{\tau}} = \nu_\tau$ for all $\tau \in \mathcal{T}_{A}$. For $n_i \in \Delta_{\tau_i} , i = 1,2 , n_1 \neq n_2$ so that $\nu(n_1) = \nu(n_2)$, we have
    \begin{align}
        \nu_{\tau_1}(l_{n_1 t_1}) \cap \nu_{\tau_2}(l_{n_2 t_2}) = \{ \nu(n_1) \}.\notag
    \end{align}
    \end{lem}

    Since we have \cref{lem2}, we can extend \cref{pro2} to the non-orientable surface case by adapting Przytycki's original proof for the orientable surface case of \cref{pro2}.

    \begin{pro}[Przytycki {\cite[Proposition 2.2]{przytycki2015arcs}}]
    \label{pro2}
     Let $F$ be a complete finite-area hyperbolic surface with at least one cusp, and let $A \in \mathscr{A}(F,1)$ be a $1$-system of arcs on $F$. We let $\Delta$ be the disjoint union of $\Delta_{\tau}$ for all tips $\tau \in \mathcal{T}_{A}$. Let $\nu : \Delta \rightarrow F$ be a map so that $\nu|_{\Delta_{\tau}} = \nu_\tau$ for all $\tau \in \mathcal{T}_{A}$, then for any point $s \in F$, then we have
        \begin{align}
            \#\nu^{-1}(s) \leq 2 (|\chi (F)|+1).\notag
        \end{align}
    \end{pro}

    \begin{proof}[Proof of \cref{pro2}]
        Let $F' := F \setminus \{s\}$. Each point $n_i \in \nu^{-1}(s)$ belongs to one of $\Delta_\tau$, which we labeled by $\Delta_{\tau_i}$. There is a unique slit $\nu|_{l_{n_i t_{\tau_i}}}: l_{n_i t_{\tau_i}} \rightarrow F$ (see \cref{defi: slits}), then $\alpha_i := \nu(l_{n_i t_{\tau_i}}) \setminus \{s\}$ is a simple arc by \cref{lem2}, which starts from $s$ and ends at a cusp of $F$. Moreover, $\alpha_i \cap \alpha_j = \{s\}, i \neq j$ by \cref{lem3}. Hence, $\{\alpha_i \mid i \}$ is a $0$-system of arcs in $F'$ whose elements start from $P_1 := \{s\}$ and end at $P_2 := \text{all cusps of\ }F$. By \cref{lem1}, $ \# \{\alpha_i \mid i \} \leq 2|\chi(F')| = 2(|\chi(F)|+1)$, where $\# \{\alpha_i \mid i \} = \# \{n_i \mid i \} = \# \nu^{-1}(s)$.
    \end{proof}

    \begin{proof}[Proof of the lower bound in \cref{thm4}]
        This proof is based on polygon decomposition. We cut $F$ along a $0$-system of arcs $A_0$ with $\# A_0 = |\chi(F)|+1$ such that $P := F \setminus \sqcup_{\alpha \in A_0}  \alpha$ is a $(2|\chi|+2)$-gon. And let $A_1$ be the set of all diagonals of $P$, where a diagonal of $P$ means a geodesic arc connecting two non-adjacent vertices of $P$. After regluing P to form $F$, the arcs $A := A_0 \sqcup A_1$ will form a $1$-system of arcs of $F$ with
        \begin{align}
            \# A = \# A_0 + \# A_1 = (|\chi (F)|+1) + \tfrac{(2|\chi (F)|+2)(2|\chi (F)|+2-3)}{2} = 2|\chi (F)|(|\chi (F)|+1).
        \end{align}
    \end{proof}

    \begin{proof}[Proof of the upper bound in \cref{thm4}]
        Let $A$ be a $1$-system of arcs on $F$. The area of the surface $F$ is $\text{Area}(F) = 2 \pi |\chi(F)|$. Since each arc in $A$ admits both positive and negative orientations, and each oriented arc corresponds to the first arc in a pair of arcs associated to a tip of $A$, therefore, $\text{Area}(\Delta)$ is the unit area of one triangle times the number of tips of $A$, which is $\pi \cdot (2\# A)$.

        By \cref{pro2}, the map $\nu : \Delta := \sqcup_{\tau} \Delta_{\tau} \rightarrow F$ is at most $2(|\chi(F)|+1)$ to $1$. Hence, we have
        \begin{align}
            \tfrac{2 \pi \# A}{2 \pi |\chi(F)| } = \tfrac{\text{Area}(\Delta) }{\text{Area}(F)} \leq 2(|\chi(F)|+1) \Rightarrow \# A \leq 2|\chi(F)|(|\chi(F)|+1). \notag
        \end{align}

    \end{proof}

\section{Complete $1$-Systems of Loops On Non-Orientable Surfaces}\label{section5}

\subsection{Cases of non-Negative Euler Characteristics}

    There are only three topological types of compact non-orientable surfaces with non-negative Euler characteristics. We will find the cardinality of maximal complete $1$-systems of loops of them case by case.

    \begin{figure}[H]
        \centering
        \includegraphics[width=1.0\textwidth]{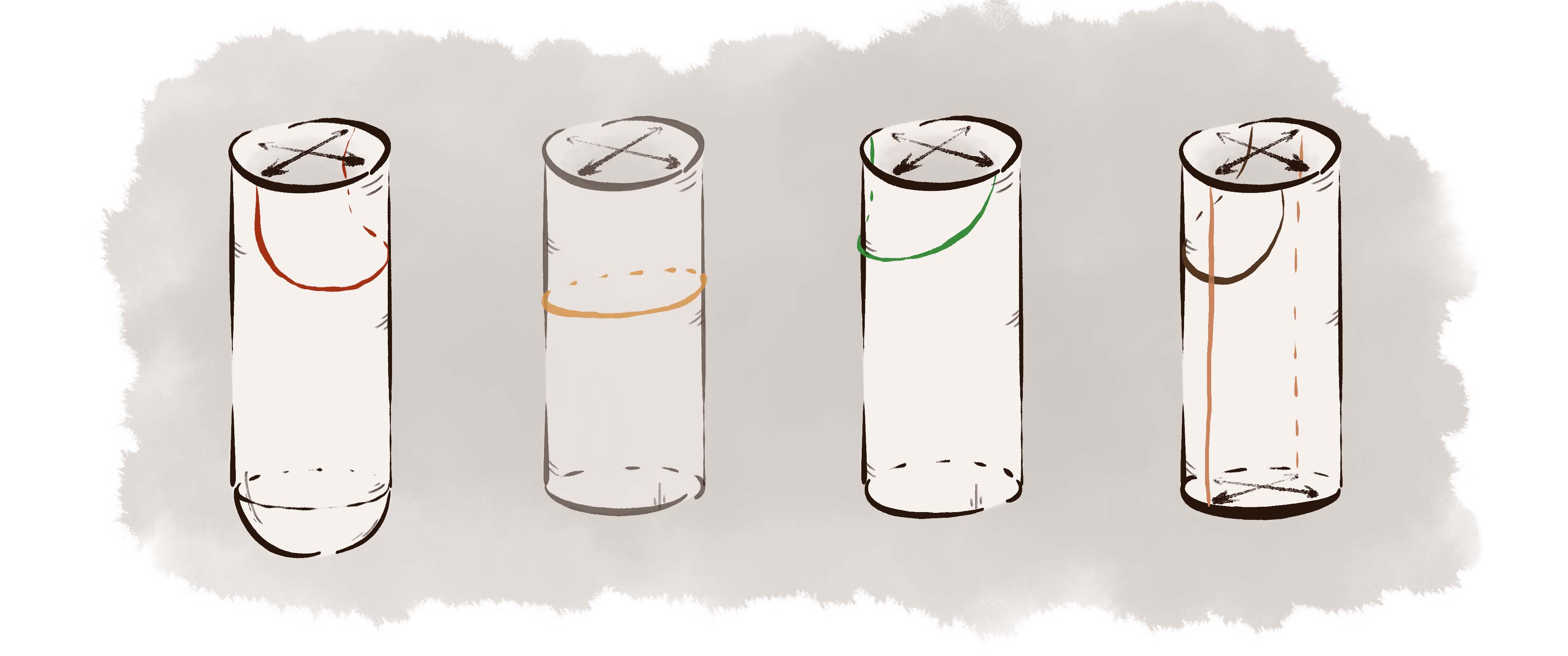}
        \caption{Maximal complete $1$-systems of loops in non-hyperbolic cases (except the second picture).}
        \label{NonHyperbolicCases}
        \end{figure}

    For $F = N_{1,0} = \mathbb{RP}^2$, since $\pi_{1}(F) = \mathbb{Z}_2$, there is only one essential simple loop up to isotopy on $F$ (see the first picture in \cref{NonHyperbolicCases}) Hence, $\|\widehat{\mathscr{L}}(N_{1,0},1) \|_\infty = 1$.

    For a Möbius strip $F = N_{1,1}$, the circle which crosses through the cross-cap is an example of an essential $1$-sided simple loop on $F$ (see the third picture in \cref{NonHyperbolicCases}). For a $2$-sided loop $\gamma$ on $F$, $W(\gamma)$ (see \cref{defi RN}) is homeomorphic to an annulus, we can only add one cross-cap to one of boundaries of $W(\gamma)$ to obtain $F$, which means $\gamma$ is isotopic to another boundary (see the second picture in \cref{NonHyperbolicCases}). Therefore, there is no essential $2$-sided simple loop on $F$. Assume that there are two essential $1$-sided simple loops $\gamma_1, \gamma_2$ intersecting once, then $W( \{ \gamma_1 , \gamma_2 \})$ is homeomorphic to $N_{1,2}$ (See \cref{TwoOneSided}). To obtain $F$, the only way is to add one disk to one of boundaries of $W( \{\gamma_1 , \gamma_2 \})$, but then these two loops will be isotopic. Hence, $\| \widehat{\mathscr{L}}(N_{1,1},1) \|_\infty = 1$.

    \begin{figure}[H]
        \centering
        \includegraphics[width=1.0\textwidth]{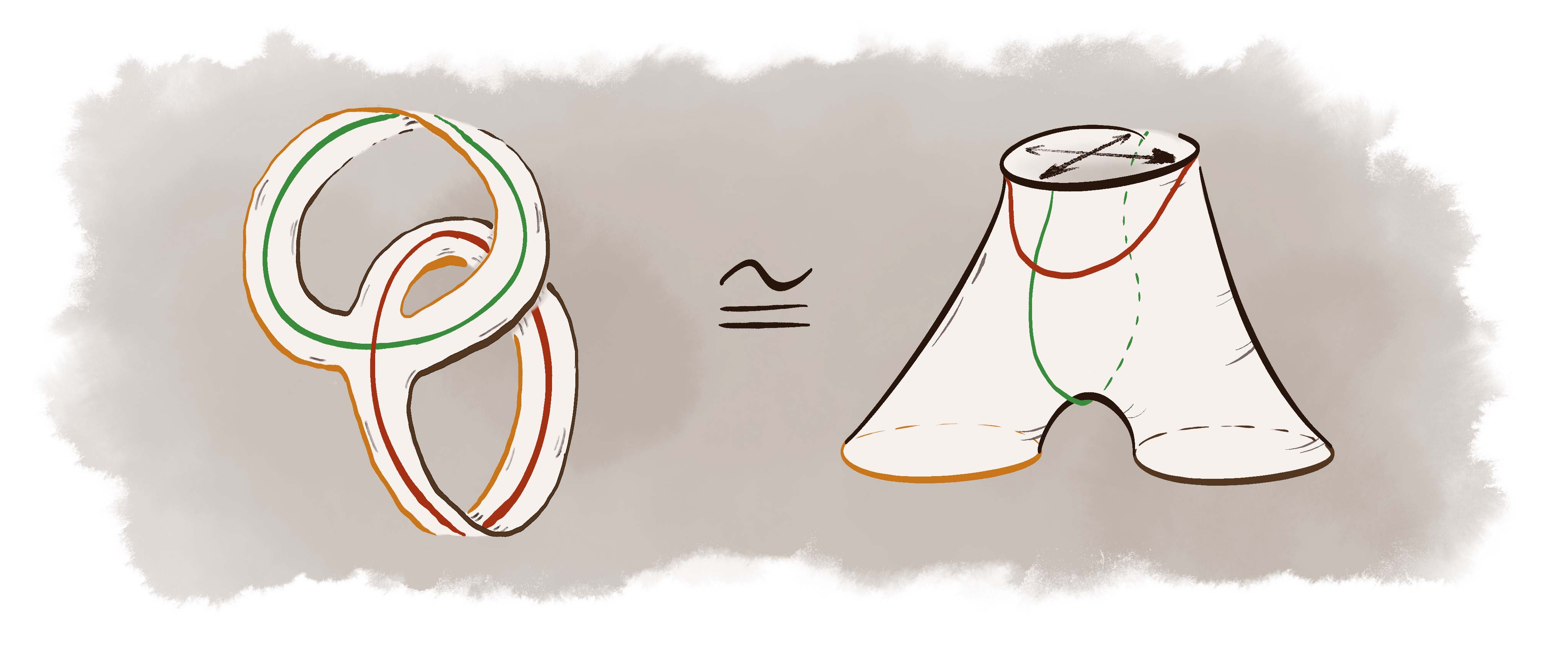}
        \caption{The regular neighbourhood of two $1$-sided loops intersecting once.}
        \label{TwoOneSided}
        \end{figure}

    For a Klein bottle $F = N_{2,0}$, 
    we consider a complete $1$-system of loops $L$ of $F$. Assume that there are two $1$-sided loops $\gamma_1, \gamma_2$ in $L$, As before, $W( \{ \gamma_1 , \gamma_2 \})$ is homeomorphic to $N_{1,2}$. Hence, after adding one disk and one cross-cap to each boundary of $W( \{ \gamma_1 , \gamma_2 \})$, $\gamma_1$ will be isotopic to $\gamma_2$. We assume that there are two $2$-sided loops $\gamma_1, \gamma_2$ in $L$, then $W( \{ \gamma_1 , \gamma_2 \})$ is homeomorphic to $S_{1,1}$, there is no way to add additional topology to the boundary of $W( \{ \gamma_1 ,\gamma_2 \})$ to obtain $F$. Therefore, $L$ at most consists of one $1$-sided loop and one $2$-sided loop, i.e. $\|\widehat{\mathscr{L}}(N_{1,0},1)\|_\infty \leq 2$. The last picture in \cref{NonHyperbolicCases} shows $L$. Hence, $\|\widehat{\mathscr{L}}(N_{2,0},1)\|_\infty = 2$.

\subsection{Lower Bound for System Cardinality}

    \begin{thm}
    \label{thm1}
        Let $F = N_{c,n}$ with $|\chi(F)| > 0$,then we have
        \begin{align}
        & \max \left\{ | L | \middle| L \in \widehat{\mathscr{L}}(F,1) \text{\ s.t. there are $t$ loops in $L$ which are $2$-sided\ } \right\} \notag \\
        & \geq
        \begin{cases}
            { {c-1+n} \choose 2 } + \left\lfloor \tfrac{c - 1 - t}{2}  \right\rfloor + t + 1 , & \text{\ if $0 \leq t \leq c - 1$\ } , \\
            c, & \text{\ if $t = c$ and $c$ is odd\ } . \\
        \end{cases}\\
        & =
        \begin{cases}
            \tfrac{1}{2} \left( (c+n)^2 - 3n - 2c + t + 2 \right) , & \text{\ if $0 \leq t \leq c - 1$ and $c-t$ is even\ } , \\
            \tfrac{1}{2} \left( (c+n)^2 - 3n - 2c + t + 3 \right) , & \text{\ if $0 \leq t \leq c - 1$ and $c-t$ is odd\ } , \\
            c , & \text{\ if $t = c$ and $c$ is odd\ } . \\
        \end{cases} \label{eq: lowerboundthm1}
        \end{align}
    \end{thm}

    \begin{proof}[Proof.]
        We set one cross-cap at the center of a sphere, and distribute the remaining cross-caps and boundaries on a circle around the center (see \cref{LowerCompleteOneSystem}).
        
        We construct a complete $1$-system of loops $L$ containing exactly $t$ loops which are $2$-sided. There are four types of loops in $L$, that is $L = \sqcup_{i=1}^{4} L_i$.

        \begin{figure}[H]
            \centering
            \includegraphics[width=1.0\textwidth]{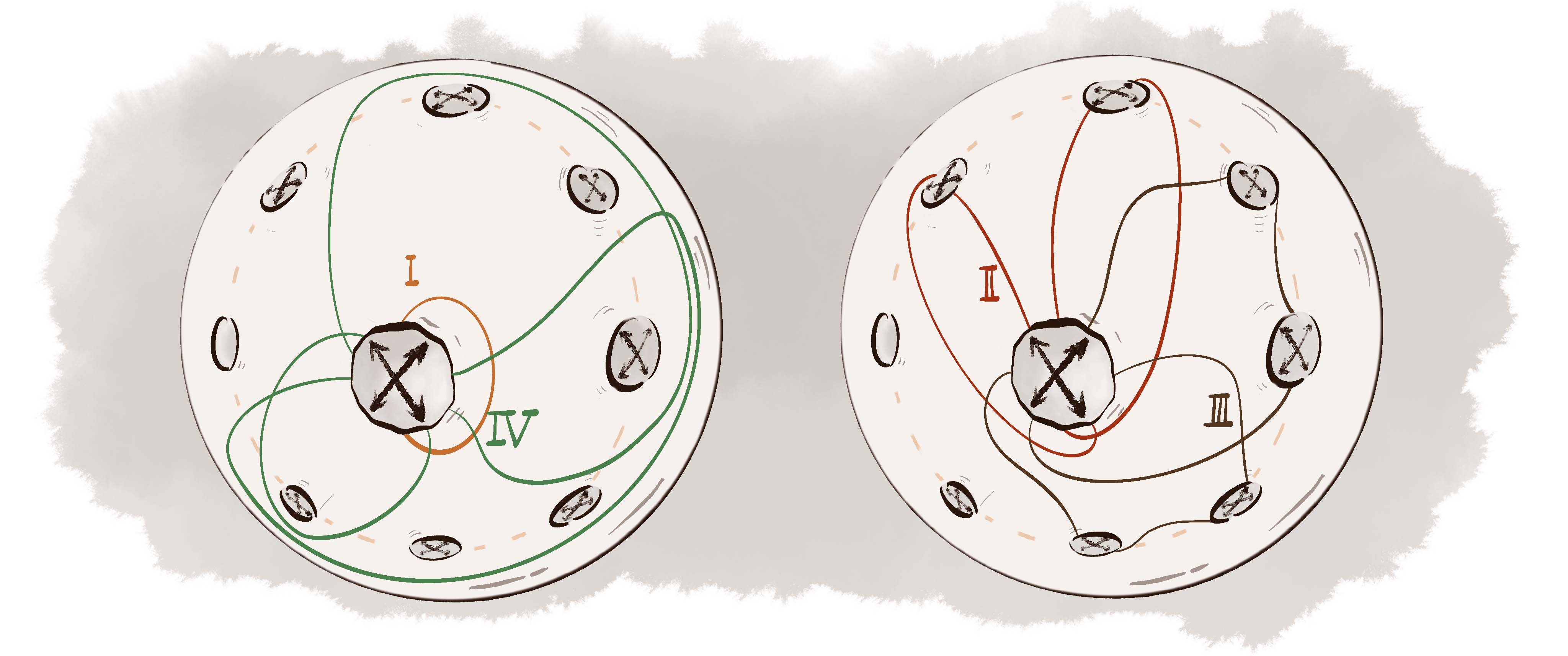}
            \caption{Examples of each type of loops in the complete $1$-system of loops on $N_{8,1}$.}
            \label{LowerCompleteOneSystem}
            \end{figure}

        \begin{itemize}
            \item There is only one type I loop, which is the unique loop whose double lift is the boundary of the central cross-cap. In fact, for any $1$-sided loop, its double lift will become the boundary of some cross-cap. Therefore, designating the type I loop serves to mark a cross-cap, which we call the \emph{central cross-cap}.
        
            \item Type II loops are those $2$-sided loops, which start from the central cross-cap, travel directly to an outer cross-cap, pass through it, and then return straight to the center. These loops pass through a chain of $t$ \emph{adjacent} outer cross-caps.
        
            \item Type III loops are the $\left\lfloor \tfrac{c - 1 - t}{2} \right\rfloor$ loops that start from the central cross-cap, pass directly through two consecutive outer cross-caps which are disjoint from all the type II loops, and then return straight to the center. There are $\left\lfloor \tfrac{c - 1 - t}{2} \right\rfloor$ pairs of adjacent cross-caps disjoint from type II loops. So we add $\left\lfloor \tfrac{c - 1 - t}{2} \right\rfloor$ type III loops.
        
            \item Type IV loops are those that start from the central cross-cap, pass through a gap between outer cross-caps and holes, travel around the outside of the circle to another gap, pass through it, and finally return to the center. Since there are a total of $c-1+n$ gaps, there are ${ {c-1+n} \choose 2 }$ type IV loops.
        \end{itemize}

        Comparing them pairwise, we can see that any two loops in $L$ are non-isotopic and intersect exactly once, i.e. they form a complete $1$-system of loops. Hence, we have
        \begin{align}
            \# L & = \# L_1 + \# L_2 + \# L_3 + \# L_4 = 1 + t + \left\lfloor \tfrac{c - 1 - t}{2}  \right\rfloor + { {c-1+n} \choose 2 }.
        \end{align}
        After simplifying, we have
        \begin{align}
        \# L =
            \begin{cases}
                \tfrac{1}{2} \left( (c+n)^2 - 3n - 2c + t + 2 \right) , & \text{\ if $0 \leq t \leq c - 1$ and $c-t$ is even\ } , \\
                \tfrac{1}{2} \left( (c+n)^2 - 3n - 2c + t + 3 \right) , & \text{\ if $0 \leq t \leq c - 1$ and $c-t$ is odd\ }.
            \end{cases}
        \end{align}

        Lastly, if $c$ is odd and $t=c$, then $F = N_{c,n}$, which is connected sum of $ S_{{\frac{1}{2}(c-1)},0}$ and $ N_{1,n}$. Hence, there are $2g+1 = 2 \cdot \tfrac{1}{2}(c-1) + 1 = c$ loops which are $2$-sided and pairwise intersect once on $S_{g,0} := S_{{\frac{1}{2}(c-1)},0}$.

    \end{proof}

\begin{rmk}[Recall \cref{cor1}]
Let $F = N_{c,n} $. The maximal cardinality of complete $1$-systems of loops on $F$ (denoted by $\| \widehat{\mathscr{L}}(F,1) \|_\infty$) satisfies:
\begin{align}
\| \widehat{\mathscr{L}}(F,1) \|_\infty \geq { {c-1+n} \choose 2 } + c = \tfrac{1}{2} \left( (c+n)^2 - 3n - c + 2 \right).\notag
\end{align}
\end{rmk}

\begin{proof}[Proof of \cref{cor1}]
One may determine the maximal value of the lower bound estimation of $\max \left\{ \# L \mid L \in \widehat{\mathscr{L}}(F,1), \text{\ there are $t$ loops in $L$ which are $2$-sided\ } \right\}$ (\cref{eq: lowerboundthm1}) with respect to $t$ in \cref{thm1} is obtained when $t = c-1$.
\end{proof}

    we similarly construct a lower bound for $\| \mathscr{L}(F,1) \|_\infty$ when $F$ is a non-orientable surface. This yields \cref{thm5}.

    \begin{proof}[Proof. of \cref{thm5}]
        We arrange cross-caps and boundaries on the circle same as the setting of the theorem before.
        
        And we also partition the $1$-system of loops $L$ into four types of loops $L_1 \sqcup L_2 \sqcup L_3 \sqcup L_4$  (see \cref{LowerOneSystem}).

        \begin{itemize}
            \item As in the proof of \cref{thm1}, the unique type I loop is the $1$-sided loop we choose to mark as the central cross-cap.

            \item Type II loops are the $2$-sided loops that start from the central cross-cap, travel directly to an outer cross-cap, pass through it, travel around the outside of the circle to a gap, pass through it, and return to the center. There are $c-1$ choices for the outer cross-caps and $c-1+n$ choices for the gaps. However, if we select a specific outer cross-cap, the two loops corresponding to the two gaps adjacent to this outer cross-cap are isotopic. Therefore, we actually only have $c+n-2$ gap options. In total, there are $(c-1)(c+n-2)$ loops in $L_2$.
        
            \item Type III loops are those that start from the central cross-cap, travel directly to an outer cross-cap, pass through it, travel around the outside of the circle to another outer cross-cap, pass through it, and then return directly to the center. Thus, there are ${ {c-1} \choose 2 }$ loops in $L_3$.
        
            \item Type IV loops are those that start from the central cross-cap, pass through a gap between outer cross-caps and holes, travel around the outside of the circle to another gap, pass through it, and finally return to the center. Since there are a total of $c-1+n$ gaps, there are ${ {c-1+n} \choose 2 }$ type IV loops. This is same construction in the proof of \cref{thm1}.
        \end{itemize}

        \begin{figure}[H]
            \centering
            \includegraphics[width=1.0\textwidth]{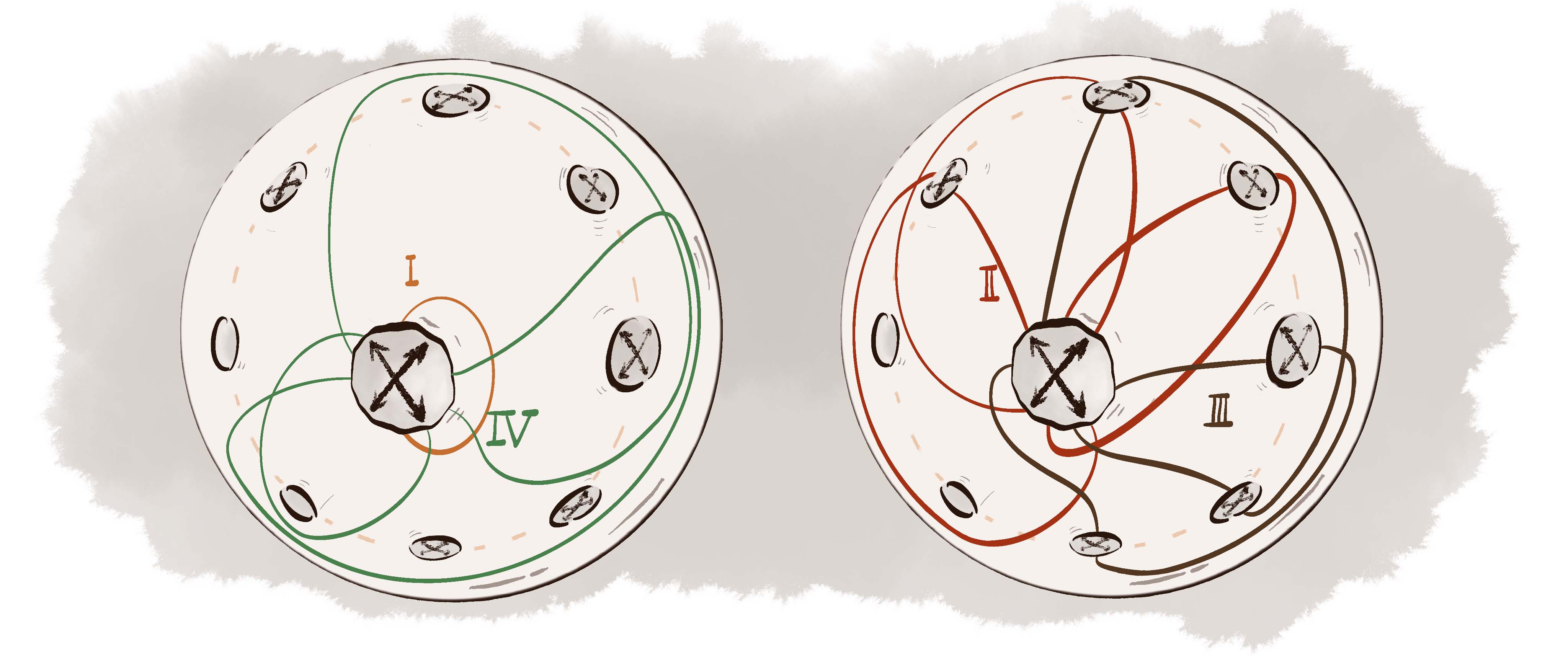}
            \caption{Examples of each type of loops in the $1$-system of loops on $N_{7,2}$.}
            \label{LowerOneSystem}
            \end{figure}

        By comparing them pairwise, we can see that any pair of loops in $L$ are non-isotopic and intersect at most once. Thus it forms a complete $1$-system of loops. Hence, we have
        \begin{align}
                \# L & = \# L_1 + \# L_2 + \# L_3 + \# L_4 = 1 + (c-1)(c+n-2) + { {c-1} \choose 2 } + { {c-1+n} \choose 2 },
        \end{align}
        as desired.

        \end{proof}

\subsection{Upper Bound for System Cardinality}

    We are going to use \cref{thm6} to prove \cref{thm2}.

    \begin{thm}[Przytycki {\cite[Theorem 1.7]{przytycki2015arcs}}]
    \label{thm6}
        Let $F = S_{0,n}$ be a punctured sphere with $\chi(F) <0$ and let $p_1$ and $p_2$ be two punctures of $F$, which are not necessarily distinct. We define $\mathscr{A}(F,1,\{p_1\}, \{p_2\})$ as the set of $1$-systems of arcs in $F$ such that for every arc $\alpha$, one of the end points of $\alpha$ is at $p_1$, and the other one is at $p_2$. Then,
        \begin{align}
            \| \mathscr{A}(F,1,\{p_1\}, \{p_2\}) \|_\infty = \tfrac{1}{2}|\chi (F)|(|\chi (F)|+1).
        \end{align}
    \end{thm}

    \begin{lem}
        \label{lem 4}
        Any $2$-sided loop on $F := N_{1,n}$ is separating.
    \end{lem}

    \begin{proof}
        Assume that $\gamma$ is a non-separating $2$-sided loop on $F$. Its regular neighborhood $W(\gamma)$ is an annulus. Since $\gamma$ is non-separating, $F \setminus W(\gamma)$ is connected. Therefore, we can find a simple arc $\alpha$ in $F \setminus W(\gamma)$, connecting the two boundaries of $W(\gamma)$. We can determine the topological type of $F' := W(\gamma) \cup W(\alpha)$ by its Euler characteristic (which is $-1$) and the number of boundary components (which is $1$). In particular, $F'$ either $S_{1,1}$ or $N_{2,1}$ depending on whether the band $W(\alpha)$ connects the boundaries of $W(\gamma)$ in an orientation-preserving or reversing manner. If $F' = N_{2,1}$, since $F'$ is a subsurface of $F$ and already has two cross-caps, the number of cross-caps of $F$ must be at least 2. This contradicts the fact that $F$ is $N_{1,n}$. If $F'= S_{1,1}$, then $F \setminus F'$ is non-orientable, otherwise, if both $F'$ and its complement are orientable, their connected sum, $F$, would also be orientable, leading to a contradiction. Thus, the surface $F \setminus F'$ must be $N_{d,n+1}$ for $d \geq 1$. Therefore, $F$ has at least one annulus (in $F'$) and one cross-cap (in the complement of $F'$). According to \cref{TwoCrossCaps}, $F$ must have at least three cross-caps, which contradicts $F = N_{1,n}$.

    \end{proof}

    \begin{recall}[\cref{thm2}]
        Let $F = N_{c,n} $. We have
            \begin{align}
            & \|\widehat{\mathscr{L}}(F,1)\|_\infty \leq
            \begin{cases}
                \tfrac{1}{2}n^2 -  \tfrac{1}{2}n +1 , & \text{\ if $c=1$\ } , \\
                2n^2 +n +2 , & \text{\ if $c=2$\ } , \\
                2|\chi (F)|^2 + 2|\chi (F)| + 1 , & \text{\ otherwise\ } . \\
            \end{cases}\notag
            \end{align}
        \end{recall}

    \begin{proof}[Proof of \cref{thm2}]
        We will sequentially prove the three cases in the statement of \cref{thm2}.

       \emph{(i)} We first consider the case when $F = N_{1,n}$.

        If $L \in \widehat{\mathscr{L}}(F,1)$ contains a separating loop, then the cardinality of $L$ is $1$. Thus, we only consider the case where $L$ consists entirely of 1-sided loops by \cref{lem 4}. Choose one of these loops and call it $\gamma$. We cut $F$ along $\gamma$ to obtain a surface $F' = S_{0,n+1}$ with an additional boundary $\gamma'$ (where $\gamma'$ is the double lift of $\gamma$). Since every loop in $L \setminus \{\gamma\}$ intersects $\gamma$ exactly once, cutting along $\gamma$ severs them into arcs, with both endpoints of each arc on $\gamma'$. We denote the set of these arcs on $F'$ by $A$ (see the right picture of \cref{N1nCase}). Since $A$ is formed by cutting loops that intersect each other exactly once, the arcs in $A$ intersect each other at most once (possibly fewer times, as they might not be in minimal position if not isotoped after cutting). This means that $A$ is a 1-system, even through it might not be complete. Hence, by \cref{thm6},
        \begin{align}
        \#A \leq \tfrac{1}{2}|\chi(F')|(|\chi(F')|+1) = \tfrac{1}{2}|2-(n+1)|(|2-(n+1)|+1) = \tfrac{1}{2}n(n-1). \notag
        \end{align}
        Thus, 
        \begin{align}
        \#L = \#A + 1 \leq \tfrac{1}{2}n(n-1) + 1. \notag
        \end{align}
        This establishes the upper bound for the case $F = N_{1,n}$.

        \begin{figure}[H]
            \centering
            \includegraphics[width=1.0\textwidth]{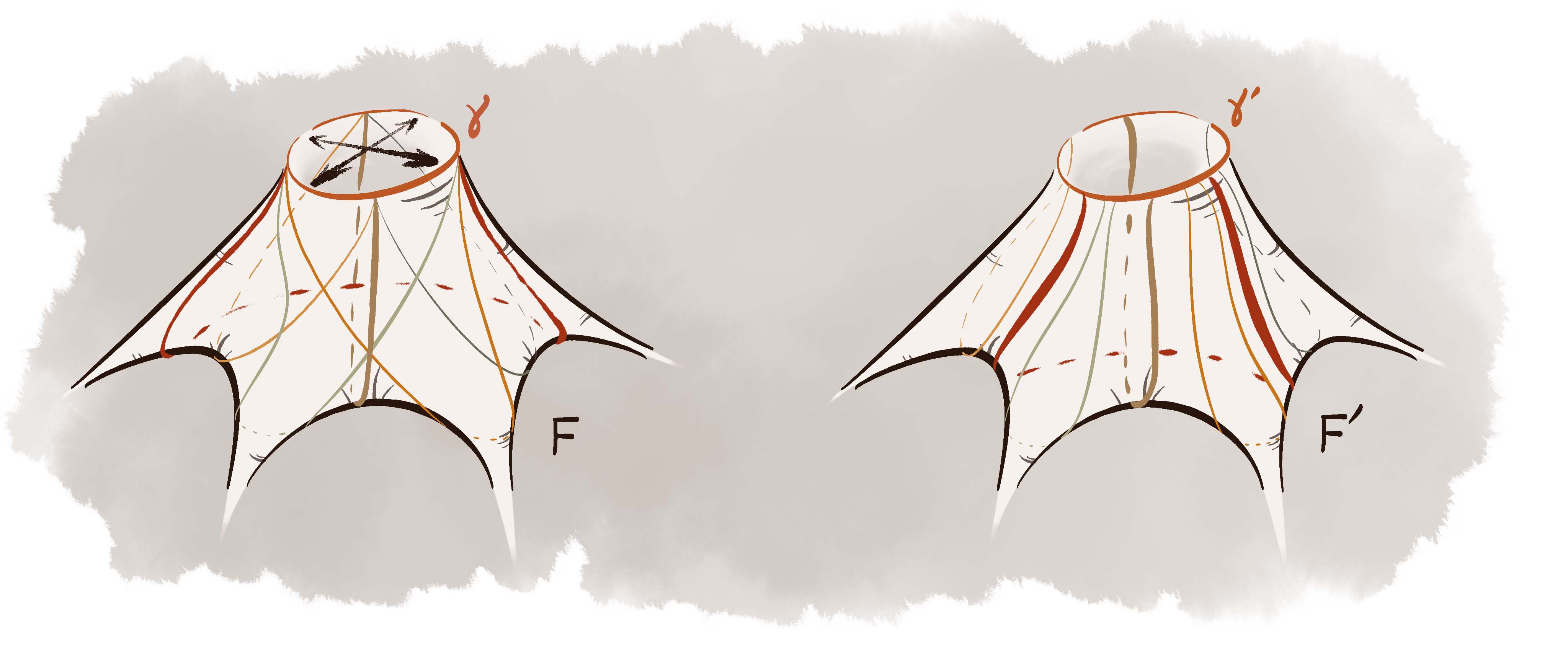}
            \caption{A maximal complete $1$-system of loops on $N_{1,4}$.}
            \label{N1nCase}
            \end{figure}
        
        \emph{(ii)} We next consider the case when $F = N_{2,n}$.

        First, we prove that in a complete $1$-system on $N_{2,n}$, there can be at most one 2-sided loop: suppose $L \in \widehat{\mathscr{L}}(F,1)$ and $L$ has more than one 2-sided loop. Consider any two such loops, $\beta_1$ and $\beta_2$, which intersect exactly once. Since they intersect exactly once and are both 2-sided, the regular neighborhood of the union of $\beta_1, \beta_2$ is a torus with one boundary component, i.e. $S_{1,1}$ (see \cref{rmk1}). Thus, $N_{2,n}$ has an orientable subsurface $S := S_{1,1}$. As the proof of \cref{lem 4}, $F$ must have at least three cross-caps, which contradicts $F = N_{2,n}$.

        Second, consider $L \in \widehat{\mathscr{L}}(F,1)$ containing at least one 1-sided loop. Choose any 1-sided loop in $L$ and call it $\gamma$. We cut $F$ along $\gamma$ to obtain a surface $F' = N_{1,n+1}$ with an additional boundary $\gamma'$. Since every loop in $L \setminus \{ \gamma \}$ intersects $\gamma$ exactly once, cutting along $\gamma$ severs the remaining loops into arcs, each with endpoints on $\gamma'$. Denote the set of these arcs on $F'$ by $A$.
        
        Next, consider the double cover $\widetilde{F'} = S_{0,2n+2} \rightarrow F'$ (see \cref{N2nCase}). Since $\gamma'$ lifts to two loops in $\widetilde{F'}$, the lifts of each arc in $A$ has four endpoints, two on each lift of $\gamma'$. Therefore, each arc in $A$ lifts to two arcs in $\widetilde{F'}$, and these lifts, denoted by $\widetilde{A}$, are also arcs. Consequently, the covering map restricts to each arc in $\widetilde{A}$ is bijective. Hence, the intersection number of any two connected components of lifts of arcs from $A$ is at most $1$, since double lifting ensures that the intersection number between any two lifts of arcs does not increase (otherwise, this would contradict the bijectivity of the covering map on the arc in $\widetilde{A}$). Therefore, the set $\widetilde{A}$, consisting of all the lifts of arcs in $A$, is a 1-system of arcs on $\widetilde{F'}$, satisfying $\# \widetilde{A} = 2 \# A$.

        Let $A_1$ be the set of arcs in $A$ which came from cutting the \emph{1-sided} loops in $L \setminus \{ \gamma \}$. For each lift of an arc in $A_1$, its start and end are on the same lift of $\gamma'$ because they pass through the other cross-caps an even number of times. By \cref{thm6},
        \begin{align}
            \# \widetilde{A_1} &\leq \sum_{i=1}^2 \max \left\{ \# A \middle| A \in \mathscr{A}(\widetilde{F'},1), \text{ for any $\alpha \in A$, both endpoints of $\alpha$ are in $\widetilde{\gamma'_i}$ } \right\} \notag \\
            &= 2 \times \tfrac{1}{2}|\chi(\widetilde{F'})|(|\chi(\widetilde{F'})|+1) = 2n(2n+1), \notag
        \end{align}
        where $\widetilde{\gamma'_1}$ and $\widetilde{\gamma'_2}$ are the two lifts of $\gamma'$. Hence, $\# A_1 = \tfrac{1}{2} \# \widetilde{A_1} \leq n(2n+1)$. Moreover, let $A_2$ be the set of arcs in $A$ which came from cutting the 2-sided loops in $L \setminus \{ \gamma \}$. From the first step of the proof of this case, we know $\# A_2 \leq 1$. Therefore, $\# L = \# A + 1 = \# A_1 + \# A_2 + 1 \leq n(2n+1) + 1 + 1 = 2n^2 + n + 2$. This provides the upper bound for the case $F = N_{2,n}$.

        \begin{figure}[H]
            \centering
            \includegraphics[width=1.0\textwidth]{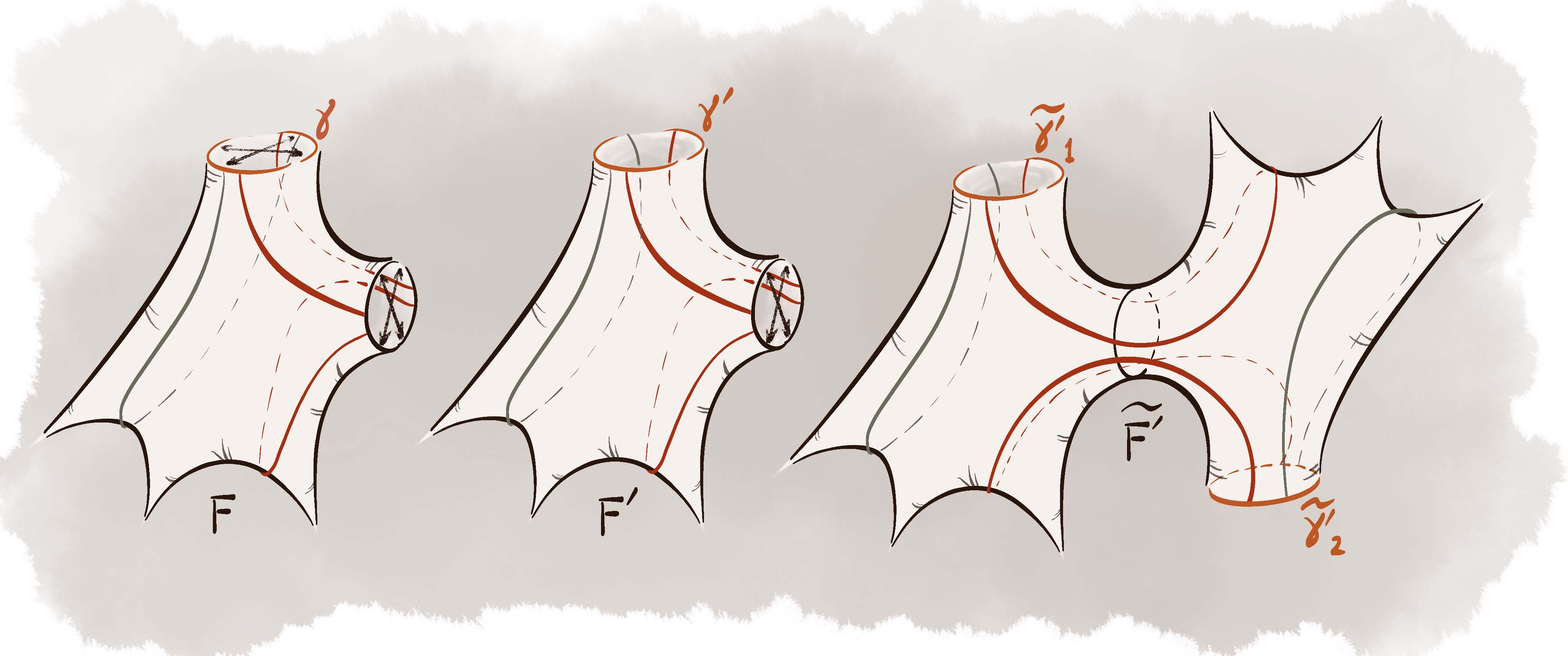}
            \caption{An example of a complete loops $1$-system of loops on $N_{2,3}$}
            \label{N2nCase}
            \end{figure}

        \emph{(iii)} Lastly, we consider the case when $F = N_{c,n}$ with $c \geq 3$.

        If $L \in \widehat{\mathscr{L}}(F,1)$ contains a separating loop, then the cardinality of $L$ is $1$.

        If $L$ consists entirely of non-separating loops. we select a loop $\gamma$ in $L$. Cutting $F$ along $\gamma$ results in a surface $F'$. Since cutting a surface along a loop does not change its Euler characteristic, i.e. $\chi(F') = \chi(F)$. Every loop in $L \setminus \{\gamma\}$ intersects $\gamma$ exactly once. Hence, cutting $F$ along $\gamma$ severs the remaining loops into arcs. We denote the set of these arcs by $A$. By \cref{thm4}, we have
            \begin{align}
                 \#L = \#A + 1 \leq 2|\chi(F')|(|\chi(F')|+1) + 1 = 2 |\chi(F)|(|\chi(F)|+1)+1. \notag
            \end{align}

        Thus, the upper bound for the case when $F = N_{c,n}$ with $c \geq 3$ is $2 |\chi(F)|(|\chi(F)|+1)+1 = 2|\chi (F)|^2 + 2|\chi (F)| + 1$.

    \end{proof}

\section*{Acknowledgement}

    At the end of this paper, I would like to express my gratitude to Yi Huang, my advisor, for his invaluable assistance and guidance. I also extend my thanks to Makoto Sakuma, as well as all the friends and colleagues in Japan who listened to my report, for their valuable suggestions and questions. Finally, I want to express my appreciation to Zhulei Pan, Wujie Shen, Ivan Telpukhovskiy and Sangsan (Tee) Warakkagun as some of the inspiration for this paper stemmed from discussions with them.

\newpage

\printbibliography[title={References}] 

@article {MR4034921,
    AUTHOR = {Greene, Joshua Evan},
     TITLE = {On loops intersecting at most once},
   JOURNAL = {Geom. Funct. Anal.},
  FJOURNAL = {Geometric and Functional Analysis},
    VOLUME = {29},
      YEAR = {2019},
    NUMBER = {6},
     PAGES = {1828--1843},
      ISSN = {1016-443X,1420-8970},
   MRCLASS = {57M15 (05C62 05D40)},
  MRNUMBER = {4034921},
MRREVIEWER = {Lorenzo\ Traldi},
       DOI = {10.1007/s00039-019-00517-0},
       URL = {https://doi.org/10.1007/s00039-019-00517-0},
}

@article{juvan1996systems,
    AUTHOR = {Juvan, Martin and Malni\v{c}, Aleksander and Mohar, Bojan},
     TITLE = {Systems of curves on surfaces},
   JOURNAL = {J. Combin. Theory Ser. B},
  FJOURNAL = {Journal of Combinatorial Theory. Series B},
    VOLUME = {68},
      YEAR = {1996},
    NUMBER = {1},
     PAGES = {7--22},
      ISSN = {0095-8956,1096-0902},
   MRCLASS = {57M15 (05C10)},
  MRNUMBER = {1405702},
       DOI = {10.1006/jctb.1996.0053},
       URL = {https://tlink.lib.tsinghua.edu.cn:443/https/443/org/doi/yitlink/10.1006/jctb.1996.0053},
}

@article{malestein2014topological,
    AUTHOR = {Malestein, Justin and Rivin, Igor and Theran, Louis},
     TITLE = {Topological designs},
   JOURNAL = {Geom. Dedicata},
  FJOURNAL = {Geometriae Dedicata},
    VOLUME = {168},
      YEAR = {2014},
     PAGES = {221--233},
      ISSN = {0046-5755,1572-9168},
   MRCLASS = {57M50 (05B99 57M15)},
  MRNUMBER = {3158040},
MRREVIEWER = {Dai\ Tamaki},
       DOI = {10.1007/s10711-012-9827-9},
       URL = {https://tlink.lib.tsinghua.edu.cn:443/https/443/org/doi/yitlink/10.1007/s10711-012-9827-9},
}

@article{nicholls2021large,
    AUTHOR = {Nicholls, Sarah R. and Scherich, Nancy and Shneidman, Julia},
     TITLE = {Large 1-systems of curves in nonorientable surfaces},
   JOURNAL = {Involve},
  FJOURNAL = {Involve. A Journal of Mathematics},
    VOLUME = {16},
      YEAR = {2023},
    NUMBER = {1},
     PAGES = {127--139},
      ISSN = {1944-4176,1944-4184},
   MRCLASS = {57K20},
  MRNUMBER = {4574884},
MRREVIEWER = {Toshio\ Saito},
       DOI = {10.2140/involve.2023.16.127},
       URL = {https://tlink.lib.tsinghua.edu.cn:443/https/443/org/doi/yitlink/10.2140/involve.2023.16.127},
}

@article{przytycki2015arcs,
    AUTHOR = {Przytycki, Piotr},
     TITLE = {Arcs intersecting at most once},
   JOURNAL = {Geom. Funct. Anal.},
  FJOURNAL = {Geometric and Functional Analysis},
    VOLUME = {25},
      YEAR = {2015},
    NUMBER = {2},
     PAGES = {658--670},
      ISSN = {1016-443X,1420-8970},
   MRCLASS = {57M50 (05B99)},
  MRNUMBER = {3334237},
MRREVIEWER = {Aleksander\ Malni\v{c}},
       DOI = {10.1007/s00039-015-0320-0},
       URL = {https://tlink.lib.tsinghua.edu.cn:443/https/443/org/doi/yitlink/10.1007/s00039-015-0320-0},
}

@Article{zbMATH07704564,
    AUTHOR = {Agrawal, Shuchi and Aougab, Tarik and Chandran, Yassin and
              Loving, Marissa and Oakley, J. Robert and Shapiro, Roberta and
              Xiao, Yang},
     TITLE = {Automorphisms of the {$k$}-curve graph},
   JOURNAL = {Michigan Math. J.},
  FJOURNAL = {Michigan Mathematical Journal},
    VOLUME = {73},
      YEAR = {2023},
    NUMBER = {2},
     PAGES = {305--343},
      ISSN = {0026-2285,1945-2365},
   MRCLASS = {57K20},
  MRNUMBER = {4584864},
MRREVIEWER = {Mustafa\ Korkmaz},
       DOI = {10.1307/mmj/20205929},
       URL = {https://doi.org/10.1307/mmj/20205929},
}

@inproceedings{harvey1981boundary,
    AUTHOR = {Harvey, William J.},
     TITLE = {Boundary structure of the modular group},
 BOOKTITLE = {Riemann surfaces and related topics: {P}roceedings of the 1978
              {S}tony {B}rook {C}onference ({S}tate {U}niv. {N}ew {Y}ork,
              {S}tony {B}rook, {N}.{Y}., 1978)},
    SERIES = {Ann. of Math. Stud.},
    VOLUME = {No. 97},
     PAGES = {245--251},
 PUBLISHER = {Princeton Univ. Press, Princeton, NJ},
      YEAR = {1981},
      OPTISBN = {0-691-08264-2},
   MRCLASS = {32G15 (57N05)},
  MRNUMBER = {624817},
MRREVIEWER = {B.\ Maskit},
}

@article{masur1998geometry,
    AUTHOR = {Masur, Howard A. and Minsky, Yair N.},
     TITLE = {Geometry of the complex of curves. {I}. {H}yperbolicity},
   JOURNAL = {Invent. Math.},
  FJOURNAL = {Inventiones Mathematicae},
    VOLUME = {138},
      YEAR = {1999},
    NUMBER = {1},
     PAGES = {103--149},
      ISSN = {0020-9910,1432-1297},
   MRCLASS = {57M50 (20F67 30F60 32G15)},
  MRNUMBER = {1714338},
MRREVIEWER = {Darryl\ McCullough},
       DOI = {10.1007/s002220050343},
       URL = {https://tlink.lib.tsinghua.edu.cn:443/https/443/org/doi/yitlink/10.1007/s002220050343},
}

@article {MR1791145,
    AUTHOR = {Masur, Howard A. and Minsky, Yair N.},
     TITLE = {Geometry of the complex of curves. {II}. {H}ierarchical
              structure},
   JOURNAL = {Geom. Funct. Anal.},
  FJOURNAL = {Geometric and Functional Analysis},
    VOLUME = {10},
      YEAR = {2000},
    NUMBER = {4},
     PAGES = {902--974},
      ISSN = {1016-443X,1420-8970},
   MRCLASS = {57M50 (30F60 32G15)},
  MRNUMBER = {1791145},
MRREVIEWER = {Darryl\ McCullough},
       DOI = {10.1007/PL00001643},
       URL = {https://tlink.lib.tsinghua.edu.cn:443/https/443/org/doi/yitlink/10.1007/PL00001643},
}

@Article{zbMATH06168116,
    AUTHOR = {Masur, Howard and Schleimer, Saul},
     TITLE = {The geometry of the disk complex},
   JOURNAL = {J. Amer. Math. Soc.},
  FJOURNAL = {Journal of the American Mathematical Society},
    VOLUME = {26},
      YEAR = {2013},
    NUMBER = {1},
     PAGES = {1--62},
      ISSN = {0894-0347,1088-6834},
   MRCLASS = {57M50},
  MRNUMBER = {2983005},
MRREVIEWER = {Martin\ Scharlemann},
       DOI = {10.1090/S0894-0347-2012-00742-5},
       URL = {https://tlink.lib.tsinghua.edu.cn:443/https/443/org/doi/yitlink/10.1090/S0894-0347-2012-00742-5},
}

@Article{zbMATH01089297,
    AUTHOR = {Ivanov, Nikolai V.},
     TITLE = {Automorphism of complexes of curves and of {T}eichm\"{u}ller
              spaces},
   JOURNAL = {Internat. Math. Res. Notices},
  FJOURNAL = {International Mathematics Research Notices},
      YEAR = {1997},
    NUMBER = {14},
     PAGES = {651--666},
      ISSN = {1073-7928,1687-0247},
   MRCLASS = {57M99 (20F38 30F60 57N05)},
  MRNUMBER = {1460387},
MRREVIEWER = {Darryl\ McCullough},
       DOI = {10.1155/S1073792897000433},
       URL = {https://tlink.lib.tsinghua.edu.cn:443/https/443/org/doi/yitlink/10.1155/S1073792897000433},
}

@Article{zbMATH03951591,
    AUTHOR = {Harer, John L.},
     TITLE = {The virtual cohomological dimension of the mapping class group
              of an orientable surface},
   JOURNAL = {Invent. Math.},
  FJOURNAL = {Inventiones Mathematicae},
    VOLUME = {84},
      YEAR = {1986},
    NUMBER = {1},
     PAGES = {157--176},
      ISSN = {0020-9910,1432-1297},
   MRCLASS = {32G15 (20F38 57N05)},
  MRNUMBER = {830043},
MRREVIEWER = {K.\ Vogtmann},
       DOI = {10.1007/BF01388737},
       URL = {https://tlink.lib.tsinghua.edu.cn:443/https/443/org/doi/yitlink/10.1007/BF01388737},
}

@book {MR2850125,
    AUTHOR = {Farb, Benson and Margalit, Dan},
     TITLE = {A primer on mapping class groups},
    SERIES = {Princeton Mathematical Series},
    VOLUME = {49},
 PUBLISHER = {Princeton University Press, Princeton, NJ},
      YEAR = {2012},
     PAGES = {xiv+472},
      OPTISBN = {978-0-691-14794-9},
   MRCLASS = {57M50 (20F36 20F65 57M07 57N05)},
  MRNUMBER = {2850125},
MRREVIEWER = {Stephen\ P.\ Humphries},
}

@article {MR3589159,
    AUTHOR = {Dahmani, Francois and Guirardel, Vincent and Osin, Denis V.},
     TITLE = {Hyperbolically embedded subgroups and rotating families in
              groups acting on hyperbolic spaces},
   JOURNAL = {Mem. Amer. Math. Soc.},
  FJOURNAL = {Memoirs of the American Mathematical Society},
    VOLUME = {245},
      YEAR = {2017},
    NUMBER = {1156},
     PAGES = {v+152},
      ISSN = {0065-9266,1947-6221},
      OPTISBN = {978-1-4704-2194-6; 978-1-4704-3601-8},
   MRCLASS = {20F65 (20F06 20F34 20F67 57M07)},
  MRNUMBER = {3589159},
MRREVIEWER = {Dominik\ Gruber},
       DOI = {10.1090/memo/1156},
       URL = {https://doi.org/10.1090/memo/1156},
}

@incollection {MR2264532,
    AUTHOR = {Ivanov, Nikolai V.},
     TITLE = {Fifteen problems about the mapping class groups},
 BOOKTITLE = {Problems on mapping class groups and related topics},
    SERIES = {Proc. Sympos. Pure Math.},
    VOLUME = {74},
     PAGES = {71--80},
 PUBLISHER = {Amer. Math. Soc., Providence, RI},
      YEAR = {2006},
      OPTISBN = {978-0-8218-3838-9; 0-8218-3838-5},
   MRCLASS = {57M07 (20F65 57M05)},
  MRNUMBER = {2264532},
MRREVIEWER = {Eric\ Michael\ Katerman},
       DOI = {10.1090/pspum/074/2264532},
       URL = {https://doi.org/10.1090/pspum/074/2264532},
}

@article {MR1722024,
    AUTHOR = {Luo, Feng},
     TITLE = {Automorphisms of the complex of curves},
   JOURNAL = {Topology},
  FJOURNAL = {Topology. An International Journal of Mathematics},
    VOLUME = {39},
      YEAR = {2000},
    NUMBER = {2},
     PAGES = {283--298},
      ISSN = {0040-9383},
   MRCLASS = {57M99},
  MRNUMBER = {1722024},
MRREVIEWER = {Bruno\ P.\ Zimmermann},
       DOI = {10.1016/S0040-9383(99)00008-7},
       URL = {https://doi.org/10.1016/S0040-9383(99)00008-7},
}

@article {MR3209702,
    AUTHOR = {Atalan, Ferihe and Korkmaz, Mustafa},
     TITLE = {Automorphisms of curve complexes on nonorientable surfaces},
   JOURNAL = {Groups Geom. Dyn.},
  FJOURNAL = {Groups, Geometry, and Dynamics},
    VOLUME = {8},
      YEAR = {2014},
    NUMBER = {1},
     PAGES = {39--68},
      ISSN = {1661-7207,1661-7215},
   MRCLASS = {57M60 (20F38)},
  MRNUMBER = {3209702},
MRREVIEWER = {Susumu\ Hirose},
       DOI = {10.4171/GGD/216},
       URL = {https://doi.org/10.4171/GGD/216},
}



    
    
    
    
    

\end{document}